\documentclass[12pt]{amsart}
\usepackage[dvipsnames]{color}
\usepackage{amsfonts,amssymb,amsmath,amscd,amstext}
\usepackage[colorlinks=true,linkcolor=blue,citecolor=blue]{hyperref}
\usepackage[utf8]{inputenc}
\usepackage{graphicx}
\usepackage{changes}
\usepackage{comment}
\usepackage{cleveref}
\usepackage{mathdots}
\usepackage[a4paper,top=2.5cm,bottom=2.5cm,left=2cm,right=2cm]{geometry}
\renewcommand{\leq}{\leqslant}
\renewcommand{\geq}{\geqslant}

\newcommand{\hhh}{{\mathcal{H}}}
\newcommand{\hn}{\mathbb{H}^n}

\newcommand{\g}[1]{g_\varepsilon\left( #1 \right)}
\newcommand{\normale}{N}
\newcommand{\no}{\normale}
\newcommand{\noh}{\normale^\hh}
\newcommand{\n}[1]{\nabla^\varepsilon_{ #1 }}
\newcommand{\gh}[1]{g_\hh\left( #1 \right)}
\newcommand{\vh}{v^\hh}

\newcommand{\rr}{{\mathbb{R}}}

\newcommand{\hh}{{\mathbb{H}}}

\newcommand{\Ru}{\mathbb{R}}

\newcommand{\Om}{\Omega}
\newcommand{\eps}{\varepsilon}

\renewcommand{\v}{v^\eps}
\newcommand{\vt}{\v_{2n+1}}

\newcommand{\scu}{\longrightarrow}

\definechangesauthor[color=red]{J}
\definecolor{champagne}{rgb}{0.97, 0.91, 0.81}

\definecolor{asparagus}{rgb}{0.53, 0.66, 0.42}
\DeclareMathOperator{\ric}{Ric}
\DeclareMathOperator{\graf}{graph}
\DeclareMathOperator{\divv}{div}
\DeclareMathOperator{\diver}{div}

\DeclareMathOperator{\spann}{span}
\DeclareMathOperator{\rank}{rank}
\newtheorem{theorem}{Theorem}[section]
\newtheorem{proposition}[theorem]{Proposition}
\newtheorem{definition}[theorem]{Definition}
\newtheorem{lemma}[theorem]{Lemma}
\newtheorem{corollary}[theorem]{Corollary}
\newtheorem*{exapp}{Example A1}
\newtheorem*{theoapp}{Theorem A2}
\newtheorem{thm}{Theorem}[section]
\newtheorem{prop}[thm]{Proposition}
\newtheorem{deff}[thm]{Definition}
\newtheorem{lem}[thm]{Lemma}
\newtheorem{cor}[thm]{Corollary}
\theoremstyle{definition}
\newtheorem*{rulingthm}{Ruling Property}
\newtheorem{es}[thm]{Example}

\newtheorem{example}[theorem]{Example}

\theoremstyle{remark}

\numberwithin{equation}{section}

\definechangesauthor[color=purple]{JP}
\definechangesauthor[color=blue]{S}


\title[A characterization of horizontally totally geodesic hypersurfaces]{A characterization of horizontally totally geodesic hypersurfaces in Heisenberg groups}
\author[A.~Pinamonti]{Andrea Pinamonti}
\address[A.~Pinamonti]{Dipartimento di Matematica, Università di Trento, via Sommarive 14, 38123 Povo (TN), Italy}
\email{andrea.pinamonti@unitn.it}

\author[S.~Verzellesi]{Simone Verzellesi}
\address[S.~Verzellesi]{Dipartimento di Matematica, Università di Trento, via Sommarive 14, 38123 Povo (TN), Italy}
\email{simone.verzellesi@unitn.it}

\date{\today}

\begin{document}
\maketitle
\begin{abstract}
    In this paper we achieve a first concrete step towards a better
understanding of the so-called \emph{Bernstein problem} in higher dimensional Heisenberg groups. Indeed, in the sub-Riemannian Heisenberg group $\hn$, with $n\geq 2$, we show that the only entire hypersurfaces with vanishing horizontal symmetric second fundamental form 
are hyperplanes. This result relies on a sub-Riemannian characterization of a higher dimensional ruling property, as well as on the study of sub-Riemannian geodesics on Heisenberg hypersurfaces.
\end{abstract}

\section{Introduction}
In the present paper we achieve a first concrete step towards a better
understanding of the so-called \emph{Bernstein problem} in higher dimensional sub-Riemannian Heisenberg groups. Namely, we reduce the solution of the latter to validity of suitable sub-Riemannian curvature estimates. The characterization of entire minimal hypersurfaces in higher dimensional sub-Riemannian Heisenberg groups is an intriguing open problem in the framework of sub-Riemannian geometry. This issue, which is typically known as \emph{Bernstein problem} in view of its Euclidean counterpart (cf. \cite{Giusti} for a complete survey of the topic), fits into the broader context of minimal hypersurfaces in sub-Riemannian structures (cf. \cite{MR2165405,MR2983199,MR2262784,MR2472175,MR4346009,pmc1,MR4314055,SGR,MR2043961,MR3445204,MR2448649,MR3276118} and references therein). The study of this and related issues is particularly relevant in the sub-Riemannian Heisenberg group $\hn$, 
since the latter constitutes a prototypical model in the setting of Carnot groups (cf. \cite{MR0657581,MR2363343}), sub-Riemannian manifolds (cf. \cite{MR3971262}), CR manifolds (cf. \cite{MR2312336}) and Carnot-Carathéodory spaces (cf. \cite{MR1421823}). 
We briefly recall that the $n$-th Heisenberg group $(\hn,\cdot)$ is $\mathbb R^{2n+1}$ endowed with the group law
\begin{equation*}
    p\cdot p'=(\bar x,\bar y,t)\cdot (\bar x',\bar y',t')=(\bar x+\bar x',\bar y+\bar y', t+t'+Q((\bar x,\bar y),(\bar x',\bar y'))),
\end{equation*}
where
\begin{equation*}
    Q((\bar x,\bar y),(\bar x',\bar y'))=\sum_{j=1}^n\left(x_j'y_j-x_jy_j'\right)
\end{equation*}
and where we denoted points $p\in\mathbb R^{2n+1}$ by $p=(z,t)=(\bar x,\bar y,t)=(x_1,\ldots,x_n,y_1,\ldots,y_n,t)$.
With this operation, $\hn$ is a Carnot group, whose associated \emph{horizontal distribution}, which we denote by $\hhh$, is generated by the left-invariant vector fields
\begin{equation*}
    X_j=\frac{\partial}{\partial x_j}+y_j\frac{\partial}{\partial t}\qquad \text{and}\qquad Y_j=\frac{\partial}{\partial x_j}-x_j\frac{\partial}{\partial t}
\end{equation*}
for $j=1,\ldots,n$. The standard sub-Riemannian structure $$(\hh^n,\hhh,\langle\cdot,\cdot\rangle)$$ is given by 
a suitable scalar product $\langle\cdot,\cdot\rangle$ on $\hhh$. 
One of the key differences between the Euclidean and the Heisenberg setting is that, as pointed out in \cite{MR1800768}, the classical Federer's notion of rectifiability in metric spaces (cf. \cite{MR0257325}) is not suitable for the Heisenberg group.
To solve this issue, the authors of \cite{MR1871966} introduced the intrinsic notion of \emph{$\mathbb H$-regular hypersurface}, together with a related notion of \emph{intrinsic rectifiability}. We recall that an $\mathbb H$-regular hypersurface is a subset of $\hn$ which can be described locally as the zero locus of a $C^1_{\mathbb H}$-function
(cf. \cite{MR1871966} for more precise definitions). A special class of $\mathbb H$-regular hypersurfaces is that of \emph{non-characteristic hypersurfaces}. 
In this setting, given a hypersurface $S$ of class $C^1$, we say that a point $p\in S$ is \emph{characteristic} as soon as
\begin{equation*}
    \hhh_p=T_p S,
\end{equation*}
otherwise we say that $p$ is \emph{non-characteristic}. In this last case, the \emph{horizontal tangent space}
\begin{equation*}
    \hhh T_p S=\hhh_p\cap T_p S
\end{equation*}
is a $(2n-1)$-dimensional vector space. The set of characteristic points of $S$ is denoted by $S_0$ and is called the \emph{characteristic set} of $S$.
 After \cite{MR1871966}, it was clear that the importance of $\mathbb H$-regular hypersurfaces went beyond rectifiability, as proved for instance from the striking differences between characteristic and non-characteristic setting in the approach to the sub-Riemannian Bernstein problem in $\mathbb H^1$. A first study of the latter was carried out by \cite{MR2165405,MR2435652} in the class of \emph{$t$-graphs} of class $C^2$. We recall that a hypersurface $S$ is a \emph{$t$-graph} whenever there exists $u:\Ru^{2n}\longrightarrow\Ru$ such that
\begin{equation*}
   \graf (u):=S=\{(\bar x,\bar y,u(\bar x,\bar y))\,:\,(\bar x,\bar y)\in \Ru^{2n}\}.
\end{equation*}
In the previous set of papers, the authors classified minimal $t$-graphs of class $C^2$ in the first Heisenberg group $\mathbb H^1$, finding examples of minimal $t$-graphs which are not planes. These results were generalized in \cite{MR2609016,MR2472175,MR2648078} to more general embedded $C^2$-hypersurfaces in $\mathbb H^1$. Moreover, as pointed out in \cite{MR2448649}, if one consider hypersurfaces with low regularity, the examples of minimal hypersurfaces which are not hyperplanes increase considerably.
However, the situation is different when considering non-characteristic hypersurfaces. In this context, a meaningful counterpart of hyperplanes in the Euclidean setting is the class of \emph{vertical hyperplanes}. Let us recall that a vertical hyperplane is a set $S$ of the form 
\begin{equation*}
    S=\{p\in\hn\,:\,\langle(\bar x,\bar y),(\bar a,\bar b)\rangle=c\},
\end{equation*}
for some $0\neq (\bar a,\bar b)\in \rr^{2n}$ and $c\in\rr$. An easy computation (cf. Section \ref{ruledsection}) shows that $S$ is non-characteristic. Moreover, every hyperplane which is not vertical is characteristic (cf. again Section \ref{ruledsection}). 
A first result in the understanding of minimal non-characteristic hypersufaces was achieved in \cite{MR2333095} in the class of \emph{intrinsic graphs} (cf. \cite{MR2333095}). Indeed, the authors showed that the only minimal intrinsic graphs defined by a $C^2$ function in $\mathbb H^1$ are vertical hyperplanes. This result was generalized in \cite{MR3406514} to the class of non-characteristic minimal $C^1$-hypersurfaces of $\mathbb H^1$, in \cite{MR3984100} to the class of minimal intrinsic graphs defined by an Euclidean Lipschitz function in $\mathbb H^1$, and in \cite{GiR} to the class of $(X,Y)$-Lipschitz surfaces in the sub-Finsler Heisenberg group $\mathbb H^1$. We point out that, as shown in \cite{MR2455341}, weakening the regularity of the defining function allows to find examples of minimal hypersurfaces which are not vertical planes even in the class of intrinsic graphs. While the Bernstein problem is well understood in $\mathbb H^1$, very few results are known in higher dimensions. On one hand, it has recently proved in (\cite{MR4069613}) that there is no rigidity in the class of smooth $t$-graphs in $\mathbb{H}^n$. On the other hand, when $n\geq 5$, there are counterexamples even in the class of smooth intrinsic graphs (cf. \cite{MR2333095,MR2472175}). The Bernstein problem for non-characteristic hypersurfaces is still open when $n=2,3,4$.
In $\mathbb H^1$, a key step consists in understanding that the non-characteristic part $S\setminus S_0$ of an area-stationary surface $S$ is foliated by \emph{horizontal line} segments in the following sense.
\begin{rulingthm}\label{galliritore}{\cite{MR2983199,MR3406514}}
    Let $S$ be an area-stationary surface of class $C^1$ in $\mathbb H^1$. Then, $S$ is foliated by horizontal line segments with endpoints in $S_0$.
    \end{rulingthm}
Here, by horizontal line, we mean an Euclidean line $\gamma$ such that
\begin{equation*}\label{horiz}
    \Dot{\gamma}(t)=\sum_{j=1}^n a_jX_j(\gamma(t))+\sum_{j=1}^n b_jY_j(\gamma(t)),
\end{equation*}
for some $a_1,\ldots,a_n,b_1,\ldots,b_n\in\Ru$ and for any $t\in\mathbb R$. The importance of this ruling property became even more evident in \cite{MR4433085}, where the author showed a Bernstein Theorem in the class of those minimal intrinsic graphs which present the aforementioned ruling property, thus without assuming any a priori regularity on the surfaces. The importance of this merely differential property can be appreciated even by a sub-Riemannian viewpoint. 
Indeed, in analogy with the Riemannian setting, the sub-Riemannian structure of $\hn$ allows to associate to the non-characteristic part of $S$ various notions of \emph{horizontal second fundamental forms}, which act on the horizontal tangent distribution $\hhh T S$. Despite these forms have been introduced and studied by many authors in different ways (cf. \cite{MR2898770,MR2401420,MR2354992,MR3385193,MR4193432,MR2401420}), and differently from the Riemannian framework, it is possible to distinguish a \emph{symmetric} form $\tilde h$ and a \emph{non-symmetric} form $h$. These forms has shown to be important in various settings, for instance to introduce the so-called \emph{horizontal mean curvature} $H$ (cf. \cite{MR2165405,MR2043961,MR2354992}), a suitable notion of \emph{horizontal umbilicity} (cf. \cite{MR3794892}) and for the study of rigid motions (cf. \cite{MR3385193}). In the particular case of $\mathbb H^1$, the vanishing of the form $h$, which coincides both with the symmetric form $\tilde h$ and with the horizontal mean curvature $H$, is equivalent to the aforementioned ruling property. In the higher dimensional case, however, $h$ and $\tilde h$ may differ, although it is in general true that the norm of $\tilde h$ is controlled by the norm of $h$ (cf. \Cref{secfundsection}). \\

The aim of the present paper is twofold. On one hand, we propose a generalization of the ruling property to higher dimensional Heisenberg groups, relating this new notion with the vanishing of the symmetric form $\tilde h$. More precisely, we will call \emph{horizontally totally geodesic} a hypersurface such that $\tilde h\equiv 0$ on its non-characteristic part. We stress that hypersurfaces for which $h\equiv 0$ are particular instances of horizontally totally geodesic hypersurfaces. On the other hand, it is not always the case that horizontally totally geodesic hypersurfaces satisfy $h\equiv 0$ (cf. \Cref{secfundsection}).
In the Riemannian framework, this name is motivated by the fact that every geodesic of a totally geodesic hypersurface is a geodesic of the ambient manifold. This last characterization allows to deduce that, in $\rr^n$, the only totally geodesic hypersurfaces are hyperplanes. The second aim of this paper is to provide an analogous result in the Heisenberg group. We stress that, at least in the non-characteristic case, hypersurfaces with $h\equiv 0$ are easily vertical hyperplanes (cf. \Cref{secfundsection}). Surprisingly, the same phenomenon continues to hold under the weaker requirement $\tilde h\equiv 0$. The main achievements of this paper can be then summarized in the following result.
\begin{theorem}[Main Theorem]\label{mainmain}
    Let $S\subseteq\hh^n$ be a hypersurface without boundary of class $C^2$. The following are equivalent:
    \begin{itemize}
       \item [$(i)$] $S$ is horizontally totally geodesic;
        \item [$(ii)$] $S$ is ruled.
    \end{itemize}
   If in addition $n\geq 2$ and $S$ is (topologically) closed, 
   then $(i)$ and $(ii)$ hold if and only if $S$ is a hyperplane.
\end{theorem}
In particular, in the non-characteristic setting, \Cref{mainmain} might constitute an important tool in order to approach the resolution of the Bernstein problem in the higher dimensional case. 
\begin{corollary}\label{corber}
    Let $S\subseteq\hh^n$ be a non-characteristic hypersurface without boundary of class $C^2$. Assume that $n\geq 2$ and that $S$ is (topologically) closed. If $S$ is horizontally totally geodesic, then $S$ is a vertical hyperplane.
\end{corollary}
 \Cref{corber} allows to reduce the complexity of the problem to the estimate of the norm of the horizontal second fundamental form $\tilde h$ associated to a minimal hypersurface. We point out that an approach based on curvature estimates for minimal hypersurfaces is already available in $\rr^n$, in view of the celebrated paper \cite{MR0423263}. 
Our approach to \Cref{mainmain} can be summarized in the following steps.
\subsection*{Introduction of the higher dimensional ruling property}
 The starting point consists in generalizing the ruling property to the higher dimensional case, which is done in \Cref{introruledsection} in two equivalent ways (cf. \Cref{locruldef}, \Cref{ruled} and \Cref{locrulisrul}). After discussing the connection between the characteristic set and this new notion (cf. \Cref{addomestica}), we show that the latter is well-behaved with respect to the intrinsic geometry of $\hn$. Namely, we prove that the class of ruled hypersurfaces is closed under the action of intrinsic dilations (cf. \Cref{dilaruled}), and the action of the so-called \emph{pseudohermitian transformations} (cf. \Cref{pseudoruled}). 
 \subsection*{Rigidity results for ruled hypersurfaces}
 Subsequently, we exploit the ruling property to provide rigidity results for some classes of hypersurfaces. Basically, we show that under some constraints on the size of the characteristic set, the higher dimensional ruling property is more rigid than the corresponding one in $\mathbb H^1$.
 \begin{theorem}\label{rulcar}
    Let $S$ be a hypersurface of class $C^1$. Assume that $n\geq 2$ and that $S$ is closed and without boundary. Assume that $S_0$ is countable and that $S$ is ruled. Then $S$ is a hyperplane.
\end{theorem} This result constitutes a first remarkable difference with $\mathbb H^1$, where there are instance of smooth ruled non-characteristic hypersurfaces which are not planes (cf. \Cref{ruledsection}). 
%
 \subsection*{Introduction of the horizontal second fundamental forms} In \Cref{secfundsection} we begin building a bridge between the aforementioned result, which is only differential in spirit, with the sub-Riemannian structure of $\mathbb H^n$. To this aim,  we formally introduce the two aforementioned second fundamental forms $\tilde h$ and $h$. highlighting the main differences between $\hh^1$ and higher dimensional Heisenberg groups. Moreover, we prove some simple formulas for the norms of $h$ and $\tilde h$ (cf. \Cref{norofh} and \Cref{normoftildeh}), which allows to relate in a quantitative way the two quantities.
 \subsection*{Ruled if and only if horizontally totally geodesic.} In view of \Cref{rulcar}, the main remaining obstacle to prove \Cref{mainmain} is to show the equivalence between the property of being horizontally totally geodesic and the ruling property. 
 \begin{theorem}\label{main1}
    Let $S$ be a hypersurface without boundary of class $C^2$. Then $S$ is ruled if and only if $S$ is horizontally totally geodesic.
\end{theorem}
 This result strongly relies on a local existence and uniqueness result for a particular geodesic-type initial value problem on the hypersurface (cf. \Cref{exungeo}). Sub-Riemannian geodesics have been extensively studied in the last years (cf. e.g. \cite{MR1867362,MR2421548,MR3642643,MR4396240,MR3205109} and references therein). Although local existence results for sub-Riemannian geodesics are available (cf. e.g. \cite{MR2421548}), it is not always the case that sub-Riemannian geodesics satisfies the standard geodesic equation
 \begin{equation*}
\nabla_{\Dot\Gamma}\Dot\Gamma=0,
 \end{equation*}
 being $\nabla$ a suitable sub-Riemannian connection (cf. \cite{MR2435652}). Therefore, we devote \Cref{geosection} to the study of the initial value problem associated to this kind of equations on hypersurfaces. The proof of \Cref{exungeo} can be reduced to the study of curves in domains of suitable \emph{intrinsic graphs}, and its main difficulty lies in the fact that that the initial value problem that we need to consider is \emph{a priori} overdetermined. Once \Cref{exungeo} is achieved, we are then in position to prove \Cref{main1}, and so, in view of the previous considerations, to conclude the proof of \Cref{mainmain}.\\

We point out that, in view of \Cref{main1}, another striking difference with the first Heisenberg group can be appreciated. Indeed, contrarily to what happens in $\hh^1$, it is easy to provide examples of minimal smooth hypersurfaces, at least in the characteristic setting, which are not horizontally totally geodesic. 
\begin{theorem}\label{controes}
   Let $n\geq 2$ and let $S:=\graf (u)$, where 
       $u(\bar x,\bar y)=\frac{1}{2}x_1^2-\frac{1}{2}y_1^2.$
   Then $S$ is a minimal smooth hypersurface which is not horizontally totally geodesic.
\end{theorem}

This set of results and considerations both provides a direct way to approach the Bernstein problem via curvature estimates, and highlights once more many interesting differences between $\mathbb H^1$ and higher dimensional Heisenberg groups. According to the authors' hope, it may give a burst in the grasp of such an interesting open problem as the Bernstein problem in this anisotropic setting.

\subsection*{Plan of the paper} In \Cref{preliminaries} we collect some basic preliminaries about the sub-Riemannian Heisenberg group.  In \Cref{introruledsection} we introduce the higher dimensional ruling property and we study some of its properties. In \Cref{ruledsection} we prove \Cref{rulcar}.
In \Cref{secfundsection} we recall the main definition and properties of the horizontal forms $h$ and $\tilde h$, and we introduce the notion of horizontally totally geodesic hypersurface. In \Cref{geosection} we introduce the relevant geodesic-type initial value problem (cf. \eqref{geoS}) and we show a local existence and uniqueness result (cf. \Cref{exungeo}). Moreover, we prove \Cref{main1}, \Cref{mainmain} and \Cref{controes}.
\subsection*{Acknowledgements} The authors thank G. Giovannardi, M. Ritoré, F. Serra Cassano, G. Tralli, D. Vittone and R. Young for fruitful conversations about these topics. The authors also thank D. Vittone for suggesting them the proof of \Cref{fireg}.

\section{Preliminaries}\label{preliminaries}
\subsection{The Heisenberg group $\hn$} In the following we denote by $T$ the left-invariant vector field $\frac{\partial}{\partial t}$. In this way $
X_1,\ldots,X_n,Y_1,\ldots,Y_n,T
$ is a basis of left-invariant vector fields. Moreover, the only nontrivial commutation relations are
\begin{equation*}
    [X_j,Y_j]=-[Y_j,X_j]=-2T
\end{equation*}
for any $j=1,\ldots,n$. A vector field which is tangent to $\hhh$ at every point is called \emph{horizontal}. For given $q\in\hn$ and $\lambda>0$, we define the \emph{left-translation} $\tau_q:\hn\longrightarrow\hn$ and the \emph{intrinsic dilation} $\delta_\lambda:\hn\longrightarrow\hn$ by
\begin{equation*}
    \tau_q(p):=q\cdot p\qquad\text{and}\qquad\delta_\lambda(p):=(\lambda z,\lambda^2 t)
\end{equation*}
for any $p=(z,t)\in\hn$ respectively. It is well known that both $\tau_q$ and $\delta_\lambda$ are global diffeomorphisms, and that $\delta_\lambda$ is a Lie group isomorphism. Moreover, we define the \emph{complex structure} $J$ by letting
\begin{equation*}
    J(X_i)=Y_i,\qquad J(Y_i)=-X_i\qquad\text{and}\qquad J(T)=0
\end{equation*}
for any $i=1,\ldots,n$, and extending it by linearity for general vector fields.
Given $p\in\hh^n$ we will often identify the vector $\sum_{j=1}^nv_jX_j|_p+v_{n+j}Y_j|_p$ with the point $(v_1,\ldots,v_{2n},0)\in\hh^n$.
The Haar measure of $\hn$ coincides with the $(2n+1)$-dimensional Lebesgue measure~$\mathcal{L}^{2n+1}$.
The homogeneity property $\mathcal{L}^{2n+1}(\delta_\lambda(E))=\lambda^Q\mathcal{L}^{2n+1}(E)$ holds for any measurable set $E\subseteq\hn$, where $Q=2n+2$ is the \emph{homogeneous dimension} of $\hn$ (cf. \cite{MR3587666}).
\subsection{Sub-Riemannian structure on $\hn$}
We let $g$ be the unique Riemannian metric on $\mathbb H^n$ which makes $X_1,\ldots,X_n,Y_1,\ldots,Y_n,T$ orthonormal. For the sake of notational simplicity, we let
\begin{equation*}
    Z_j=X_j,\qquad Z_{n+j}=Y_j\qquad\text{and}\qquad Z_{2n+1}=T
\end{equation*}
for any $j=1,\ldots,n$, and we recall for the sake of completeness that the horizontal distribution $\hhh$ is defined by
\begin{equation*}
    \hhh_p=\spann\{Z_1|_p,\ldots,Z_{2n}|_p\}
\end{equation*}
for any $p\in\hn$. If we restrict $g$ to the horizontal distribution $\hhh$, and we denote this restriction by $\langle\cdot,\cdot\rangle$, then $\hn$ inherits a \emph{sub-Riemannian structure} which realizes it as a \emph{sub-Riemannian manifold.} 
We denote by $|\cdot|$ the norm induced by $\langle\cdot,\cdot\rangle$. Moreover, we denote by $\nabla$ the so-called \emph{pseudohermitian connection} (cf. e.g. \cite{MR4193432}), i.e. the unique \emph{metric connection} (cf. \cite{MR1138207}) with torsion tensor given by
\begin{equation}\label{pseudotorsion}
    \nabla_X Y-\nabla_Y X-[X,Y]=2\langle J(X),Y\rangle T
\end{equation}
for any pair of vector fields $X$ and $Y$.
The most relevant feature of $\nabla$ (cf. \cite{MR2898770}) is the following property:
\begin{equation}\label{phflat}
    \nabla_{Z_i}Z_j=0
\end{equation}
for any $i,j=1,\ldots,2n+1$, and so can be seen as a flat connection on $\hh^n$. 
\subsection{Carnot-Carathéodory structure on $\hn$}
If $\Gamma:[a,b]\scu \hn$ is an absolutely continuous curve, we say that it is \emph{horizontal} whenever
\begin{equation}\label{horiz}
 \Dot\Gamma(t)\in\hhh_{\Gamma(t)}
\end{equation}
for almost every $t\in [a,b]$,
and we say that it is \emph{sub-unit} whenever it is horizontal with $|\Dot\Gamma(t)|=1$ for almost every $t\in[a,b]$.
Moreover, we define 
\begin{equation*}
    d(p,p'):=\inf\{T\,:\,\Gamma:[0,T]\scu \mathbb{H}^n\text{ is sub-unit, $\Gamma(0)=p$ and $\Gamma(T)=p'$}\}
\end{equation*}
which, by the Chow-Rashevskii theorem (cf. \cite{MR0001880}), defines a distance on $\mathbb{H}^n$, called \emph{Carnot-Carathéodory distance}. The metric space $(\hn,d)$ is then a prototype of \emph{Carnot-Carathéodory space} (cf. \cite{MR1421823}).
\subsection{Horizontal perimeter and horizontal gradient}  If $\Om\subseteq\hn$ is open and $E\subseteq \hn$ is measurable with $\chi_E\in L^1_{loc}(\Om)$, we recall (cf. e.g. \cite{MR1871966, MR1404326}) that the \emph{$\mathbb H$-perimeter} of $E$ in $\Om$ is defined by
\begin{equation*}
    P_{\mathbb H}(E,\Om):=\sup\left\{\int_E\divv_{\mathbb H}(\bar\varphi)\,d\mathcal{L}^{2n+1}\,:\,\bar\varphi\in C^1_c(\Om,\mathcal{H}),\,|\bar\varphi|_p\leq 1\text{ for any }p\in\Om\right\},
\end{equation*}
where by $C^1_c(\Om,\hhh)$ we denote the class of compactly supported $C^1$ sections of the horizontal distribution $\hhh$, and $\divv_{\mathbb H}$ is the so-called \emph{horizontal divergence}, defined by 
\begin{equation*}
    \divv_{\mathbb H}\left(\sum_{j=1}^n(\varphi_jX_j+\varphi_{n+j}Y_j)\right):=\sum_{j=1}^n(X_j\varphi_j+ Y_j\varphi_{n+j})
\end{equation*}
for any $\sum_{j=1}^n(\varphi_jX_j+\varphi_{n+j}Y_j)\in C^1(\Om,\hhh)$. Moreover, we say that a set $E$ as above is an \emph{$\mathbb H$-Caccioppoli set} whenever $P_{\mathbb H}(E,\Om)<+\infty$ for any bounded open set $\Om\subseteq\hn$. Finally, we recall (cf. e.g. \cite{MR3587666}) that an $\mathbb H$-Caccioppoli set $E$ is an \emph{$\mathbb H$-perimeter minimizer} whenever
\begin{equation*}
    P_{\mathbb H}(E,\Om)\leq P_{\mathbb H}(F,\Om)
\end{equation*}
for any $\Om \Subset\hn$ and for any $\mathbb H$-Caccioppoli set $F$ such that $E\Delta F\Subset\Om$. The sub-Riemannian structure of $\hn$ allows to define a distributional notion of \emph{horizontal gradient} (cf. \cite{MR1871966}). More precisely, if $f\in L^1_{loc}(\Om)$, we let 
\begin{equation*}
    \langle \nabla_\hh f,\varphi\rangle:=-\int_{\Omega} f \divv_\hh\varphi \,d\mathcal{L}^{2n+1}
\end{equation*}
for any $\varphi\in C^\infty_c(\Om,\rr^{2n})$. When $f$ is continuous and $\nabla_\hh f$ is represented by a continuous vector field, then we say that $f\in C^1_\hh(\Om)$. Moreover, in this case,
\begin{equation*}
    \nabla_\hh f=\sum_{j=1}^nX_jfX_j+\sum_{j=1}^nY_jfY_j.
\end{equation*}
\subsection{Hypersurfaces in $\mathbb H^n$}
We say that $S\subseteq\mathbb H^n$ is an \emph{$\mathbb H$-regular hypersurface} if, for any $p\in S$, there exists an open neighborhood $U$ of $p$ and a function $f\in C^1_{\mathbb H}(U)$ such that
\begin{equation*}
    S\cap U=\{q\in\mathbb H^n\,:\,f(q)=0\}\qquad\text{and}\qquad\nabla_{\mathbb H}f\neq 0\text{ on }U.
\end{equation*}
Here and in the rest of the paper, whether not specified, a hypersurface will be always at least of class $C^1$ and without boundary. If $S$ is a hypersurface of class $C^1$, we define 
\begin{equation*}
    S_0:=\{p\in S: \hhh_p=T_pS\}
\end{equation*}
and we call it the \emph{characteristic set} of $S$. Notice that, since $S$ is of class $C^1$ and $\hhh$ is a smooth distribution, then $S_0$ is closed in $S$. Moreover, let us define
\begin{equation*}
    \hhh T_pS:=\hhh_p\cap T_pS.
\end{equation*}
When $p\in S_0$, then $\dim(\hhh T_pS)=2n$. On the contrary, when $p\in S\setminus S_0$, we have $\dim(\hhh T_pS)=2n-1$. In this case, the \emph{horizontal normal} to $S$ at $p$ is defined by
\begin{equation}\label{hornormnonchar}
    \vh(p):=\frac{N^{\mathbb H}(p)}{|N^{\mathbb H}(p)|_p}
\end{equation}
for any $p\in S\setminus S_0$, where $N_{\mathbb H} (p)$ is the a section of the horizontal bundle defined by
\begin{equation*}
    N^{\mathbb H}(p):=\sum_{j=1}^n(\langle N(p),X_j|_p\rangle_{\rr^{2n+1}}X_j|_p+\sum_{j=1^n}\langle N(p),Y_j|_p\rangle_{\rr^{2n+1}}Y_j|_p,
\end{equation*}
being $N(p)$ the Euclidean unit normal to $S$ at $p$. Roughly speaking, $\vh$ is a unit horizontal vector field which is orthogonal to the horizontal part of the tangent space. It is clear that a hypersurface  of class $C^1$ with empty characteristic set is $\mathbb H$-regular (cf. e.g.  \cite{MR4069613}).
When $S$ is of class $C^2$, then
\begin{equation}\label{propvh1}
    \sum_{h=1}^{2n}\vh_hZ_k(\vh_h)=0
\end{equation}
for any $k=1,\ldots,2n$, where by $\vh$ we mean any $C^2$ extension of $\vh|_S$ in a neighborhood of $S$ such that $|\vh|\equiv 1$. Indeed, \eqref{propvh1} follows by taking derivatives of $|\vh|^2\equiv 1$. There is a particular choice of such an extension which allows to derive further relations. Indeed, if we let $d^\hh$ be the signed Carnot-Carathéodory distance from $S$ with respect to $(Z_1,\ldots,Z_{2n})$, then it is well known (cf. \cite{MR4193432}) that $d^\hh$ inherits the same regularity of $S$ in a neighborhood of any non-characteristic point $p\in S$. Moreover, since $d^\hh$ satisfies the horizontal Eikonal equation in a neighborhood of $S$, meaning that $|\nabla^\hh d^\hh|\equiv 1$ in a neighborood of any non-characteristic point $p\in S$ (cf. \cite{MR4193432}),
then $\vh|_S$ can be extended by letting
\begin{equation}\label{normcondist}
    \vh=\sum_{j=1}^nX_jd^\hh X_j+\sum_{j=1}^nY_jd^\hh Y_j.
\end{equation}
With this particular extension,
\begin{equation}\label{propvh2}
    Z_k(\vh_h)=Z_h(\vh_k)
\end{equation}
for any $h,k=1,\ldots,2n$ such that $|h-k|\neq n$. Moreover,
\begin{equation}\label{propvh3}
    X_k(\vh_{n+k})=Y_k(\vh_k)-2Td^\hh\qquad\text{and}\qquad Y_k(\vh_k)=X_k(\vh_{n+k})+2Td^\hh
\end{equation}
for any $k=1,\ldots,n$. Finally, thanks to \eqref{propvh1}, \eqref{propvh2} and \eqref{propvh3}, we see that
\begin{equation}\label{propvh5}
    \sum_{h=1}^{2n}\vh_hZ_h(\vh_k)=-2T d^\hh J(\vh)_k
\end{equation}
for any $k=1,\ldots,2n$.
Moreover, a simple computation shows that
\begin{equation}\label{propvh4}
  Td^\hh=\frac{\no_{2n+1}}{|\noh|}
\end{equation}
According to \cite{MR2354992}, we define the \emph{horizontal tangential derivatives}
\begin{equation*}
    \delta_i\xi=Z_i\bar\xi-\gh{\nabla_\hh,\bar\xi}\vh_i
\end{equation*}
for any $i=1,\ldots,2n$, where $\xi$ is a $C^1$ function on an open subset of $S$ and $\bar\xi$ is any $C^1$ extension of $\xi$. As customary, the horizontal tangential derivatives do not depend on the chosen extension (cf. \cite{MR2354992}). 
We recall that the \emph{tangent pseudohermitian connection} $\nabla^S$ is defined in the non- characteristic part of $S$ by 
\begin{equation*}
    \nabla^S _X Y=\nabla_XY-\langle\nabla_XY,\vh\rangle\vh
\end{equation*}
for any pair of tangent horizontal vector fields $X$ and $Y$. Notice that $\nabla^S$ is easily a metric connection with respect to the metric $\langle\cdot,\cdot\rangle$ restricted to $\hhh T S$.
We say that a hypersurface of class $C^1$ is \emph{$\hh$-minimal} whenever it coincides with the boundary of an $\mathbb H$-perimeter minimizer. 
\subsection{Intrinsic graphs} We follow the notation and the approach of \cite{MR2223801}.
Let us denote points $q\in\rr^{2n}$ by $q=(\xi_1,\ldots,\xi_n,\eta_2,\ldots,\eta_n,\tau)=(\bar\xi,\tilde\eta,\tau)$. We wish to identify $\rr^{2n}$ with $\{p\in\hn\,:\,y_1=0\}$. To this aim, we introduce the immersion map $i:\rr^{2n}\longrightarrow\hn$ defined by
\begin{equation*}
    i(\bar\xi,\tilde\eta,\tau)=(\bar\xi,0,\tilde\eta,\tau)
\end{equation*}
for any $(\bar\xi,\tilde\eta,\tau)\in\rr^{2n}$. Moreover, we identify $\rr$ with $\{(\bar 0,y_1,\tilde 0,0)\,:\,y_1\in\rr\}$ by means of the inclusion $j:\rr\longrightarrow\hn$ defined by
\begin{equation*}
    j(y_1)=(\bar 0,y_1,\tilde 0,0)
\end{equation*}
for any $y_1\in\rr$.
The maps $i$ and $j$ are clearly smooth, injective and open. For a given open set $\Om\subseteq\rr^{2n}$ and a function $\varphi:\Om\longrightarrow \rr$, we recall that the \emph{$Y_1$-graph of $\varphi$ on $\Om$} is defined by 
\begin{equation*}
    \graf_{Y_1}(\varphi,\Om)=\{i(w)\cdot j(\varphi(w))\,:\,w\in\Om\}=\{(\bar\xi,\varphi(\bar\xi,\tilde\eta,\tau),\tilde\eta,\tau-\xi_1\varphi(\bar\xi,\tilde\eta,\tau))\,:\,(\bar\xi,\tilde\eta,\tau)\in\Om\}.
\end{equation*}
Moreover, we define its parametrization map $\Psi:\Om\longrightarrow\hh^n$ by
\begin{equation*}
    \Psi(\bar\xi,\tilde\eta,\tau)=(\bar\xi,\varphi(\bar\xi,\tilde\eta,\tau),\tilde\eta,\tau-\xi_1\varphi(\bar\xi,\tilde\eta,\tau))
\end{equation*}
for any $(\bar\xi,\tilde\eta,\tau)\in\Om$. We introduce also the \emph{intrinsic projection map} $\Pi:\hh^n\longrightarrow\rr^{2n}$ by
\begin{equation*}
    \Pi(\bar x,\bar y,t)=(\bar x,\tilde y,t+x_1y_1)
\end{equation*}
for any $(\bar x,\bar y,t)\in\hn$. It is easy to check that
\begin{equation*}
    \Pi(\Psi(q))=q\qquad\text{and}\qquad\Psi(\Pi(p))=p
\end{equation*}
for any $q\in\Om$ and any $p\in \graf_{Y_1}(\varphi,\Om)$.
If $\varphi\in C^1(\Om)$ and $S=\graf_{Y_1}(\varphi,\Om)$, then 
\begin{equation*}
    T_{\Psi(q)}S=\spann\left(\frac{\partial\Psi}{\partial\xi_1}\Big|_{q},\ldots,\frac{\partial\Psi}{\partial\xi_n}\Big|_{q},\frac{\partial\Psi}{\partial\eta_2}\Big|_{q},\ldots,\frac{\partial\Psi}{\partial\eta_n}\Big|_{q},\frac{\partial\Psi}{\partial\tau}\Big|_{q}\right).
\end{equation*}
Letting $D\varphi=(\varphi_{\xi_1},\ldots\varphi_{\xi_n},\varphi_{\eta_2},\ldots,\varphi_{\eta_n},\varphi_\tau)$, an easy computation shows that
\begin{equation*}
  \frac{\partial\Psi}{\partial\xi_1}\Big|_{q}=X_1|_{\Psi(q)}+\varphi_{\xi_1}(q)Y_1|_{\Psi(q)}-2\varphi(q)T|_{\Psi(q)},
  \end{equation*}
\begin{equation*}
\frac{\partial\Psi}{\partial\xi_j}\Big|_{q}=X_j|_{\Psi(q)}+\varphi_{\xi_j}(q)Y_1|_{\Psi(q)}-\eta_jT|_{\Psi(q)},\qquad
\frac{\partial\Psi}{\partial\eta_j}\Big|_{q}=Y_j|_{\Psi(q)}+\varphi_{\eta_j}(q)Y_1|_{\Psi(q)}+\xi_jT|_{\Psi(q)}
\end{equation*}
for any $j=2,\ldots,n$ and
\begin{equation*}
  \frac{\partial\Psi}{\partial\tau}\Big|_{q}=\varphi_{\tau}(q)Y_1|_{\Psi(q)}+T|_{\Psi(q)}.
\end{equation*}
It is easy to check that $(E_1,\ldots,E_n, F_2,\ldots,F_n)$ constitutes a global basis of $\hhh TS$, where
\begin{equation*}
    E_1=X_1+W^\varphi\varphi Y_1,\quad E_j=X_j+\tilde X_j\varphi Y_1\quad\text{and}\quad F_{j}=Y_j+\tilde Y_j\varphi Y_1
\end{equation*}
for any $j=2,\ldots,n$, and where the family of vector fields $\nabla^\varphi=(W^\varphi\varphi,\tilde X_2,\ldots\tilde X_n,\tilde Y_2,\ldots,\tilde Y_n)$ is defined by 
\begin{equation*}
    W^\varphi=\frac{\partial}{\partial \xi_1}+2\varphi \tilde T,\quad \tilde X_j=\frac{\partial}{\partial \xi_j}+\eta_j\tilde T\quad\text{and}\quad \tilde Y_j=\frac{\partial}{\partial \eta_j}-\xi_j\tilde T
\end{equation*}
for any $j=2,\ldots,n$, where we have set $\tilde T=\frac{\partial}{\partial \tau}$. Therefore, a quick computation implies that
\begin{equation}\label{normygraph}
    \vh=W^{-\frac{1}{2}}\left(W^\varphi\varphi X_1+\sum_{j=2}^n\tilde X_j\varphi X_j-Y_1+\sum_{j=2}^n\tilde Y_j\varphi Y_j\right),
\end{equation}
where we have set
\begin{equation*}
    W=1+|\nabla^\varphi\varphi|^2.
\end{equation*}
Notice that, since 
\begin{equation*}
    [\tilde X_j,\tilde Y_j]=-2\tilde T
\end{equation*}
for any $j=2,\ldots,n$, then $(\Om,d_\varphi)$ is a \emph{Carnot-Carathéodory space} (cf. \cite{MR1421823}) where $\Om$ is any domain of $\rr^{2n}$ and $d_\varphi$ is the \emph{Carnot-Carathéodory distance} induced by $\nabla^\varphi$.

\section{Higher dimensional ruled hypersurfaces}\label{introruledsection}
As already mentioned, a key step in the study of minimal surfaces in $\hh^1$ consists in showing that the non-characteristic part of an area-stationary surface is foliated by horizontal line segments. This property extends to the higher dimensional case as follows.
\begin{definition}[Local ruling property]\label{locruldef}
    Let $S$ be a hypersurface of class $C^1$.  We say that $S$ is \emph{locally ruled at $p\in S\setminus S_0$} if there exists a neighborhood $U$ of $p$ such that
\begin{equation*}
    p\cdot \hhh T_p S\cap U\subseteq S.
\end{equation*}
Moreover, we say that $S$ is \emph{locally ruled} if it is locally ruled at $p\in 
 S\setminus S_0$ for any $p\in S\setminus S_0$.
\end{definition}
Beside this local definition, we propose a global one, which will be useful in the following. 
\begin{deff}[Global ruling property]\label{ruled}
Let $S$ be a hypersurface of class $C^1$. We say that $S$ is \emph{ruled} if for any $p\in S\setminus S_0$, for any $v\in \hhh T_p S$ and for any $s\geq 0$, the following property holds. If $s$ is maximal with the property that 
\begin{equation*}
    p\cdot\delta_{\tau}(v)\in S
\end{equation*}
for any $\tau\in[0,s]$,
then
\begin{equation*}
    p\cdot\delta_{s}(v)\in S_0.
\end{equation*}
\end{deff}
The previous two definitions are actually equivalent.
\begin{proposition}\label{locrulisrul}
    Let $S$ be a hypersurface of class $C^1$. Then the following are equivalent.
    \begin{itemize}
        \item [$(i)$] $S$ is locally ruled.
        \item [$(ii)$] $S$ is ruled. 
    \end{itemize}
\end{proposition}
\begin{proof}
       Assume that $S$ is ruled. Assume by contradiction that there exists $p\in S\setminus S_0$ and a sequence $(p_h)_h\subseteq p\cdot \hhh T_p S\setminus S$ converging to $p$ as $h\to+\infty$. Then, for any $h\in\mathbb N$, there exists $\lambda_h>0$ and $v_h\in \hhh T_p S$ such that $p_h=p\cdot\delta_{\lambda_h}(v_h)$. If, up to a subsequence, for any $h$ there exists $0<\mu_h\leq\lambda_h$ such that $p\cdot\delta_{\mu_h}(v_h)$ belongs to the manifold boundary of $S$, then, being the latter closed, so does $p$, a contradiction with $p\in S$. Therefore, since $p_h\notin S$ and up to a subsequence, we can assume that for any $h$ there exists $s_h\geq 0$ maximal such that $p\cdot\delta_\tau(v_h)\in S$ for any $\tau\in [0,s_h]$. Clearly $s_h\leq\lambda_h$. Therefore, being $S$ ruled, then $q_h:=p\cdot\delta_{s_h}(v_h)\in S_0$. But then, by construction, $(q_h)_h$ converges to $p$ as $h\to+\infty$, and so, being $S_0$ closed, we conclude that $p\in S_0$, a contradiction. On the contrary, assume that $S$ is locally ruled.  
    Assume by contradiction that that there exists $p\in S\setminus S_0$, $w\in\hhh T_p S$ and $s$ maximal with the property that 
\begin{equation}\label{maximaleeee}
    p\cdot\delta_{\tau}(w)\in S
\end{equation}
for any $\tau\in[0,s]$ and \begin{equation*}
        p\cdot(sw,0)\notin S_0.
    \end{equation*}
    Set $\bar p:=p\cdot(sw,0)$.
    Consider the left-invariant vector field
    \begin{equation*}
        W=\sum_{j=1}^{2n}w_jZ_j.
    \end{equation*}
    Recalling that left-translations preserve the horizontal distribution, and being $W$ left invariant, we conclude that 
    \begin{equation*}
        d\tau_{(sw,0)}|_p(W_p)=W|_{\bar p}\in \hhh_{\bar p}.
    \end{equation*}
    Moreover, by construction, $W$ is clearly tangent to $S$ at $\bar p$. We conclude that $w\in \hhh T_{\bar p} S$.
     Since $S$ is locally ruled and $\bar p\in S\setminus S_0$, there exists $\tilde s>0$ such that $\bar p\cdot(\tilde sw,0)\in S$, which implies, recalling the definition of $\bar p$, that
    \begin{equation*}
        \bar p\cdot(\tilde sw,0)=p\cdot(\bar sw,0)\cdot(\tilde sw,0)=p\cdot((\bar s+\tilde s)w,0)\in S,
    \end{equation*}
    a contradiction with \eqref{maximaleeee}.
\end{proof}

\begin{prop}\label{addomestica}
Let $S$ be a ruled hypersurface of class $C^1$. Assume that $S$ is (topologically) closed. Let $p\in S\setminus S_0$ and $v\in\hhh T_p S$ be such that
\begin{equation*}
    \left\{p\cdot\delta_s(v,0)\,:\,s\geq 0\right\}\cap S_0=\emptyset.
\end{equation*}
Then 
\begin{equation*}
     \left\{p\cdot\delta_s(v,0)\,:\,s\geq 0\right\}\subseteq S.
\end{equation*}
In particular, if
\begin{equation*}
    p\cdot \hhh T_p S\cap S_0=\emptyset,
\end{equation*}
then 
\begin{equation}\label{mergestrongruled}
    p\cdot \hhh T_p S\subseteq S.
\end{equation}
\end{prop}
\begin{proof}
Let $p\in S\setminus S_0$ and $v\in\hhh T_p S$ be as in the statement, and assume by contradiction that there exists $\lambda>0$ such that $q=p\cdot\delta_{\lambda}(v)\notin S$. Being $S$ closed, there exists $s\geq 0$ maximal as in \Cref{ruled}. Then we can argue as in the proof of \Cref{locrulisrul} to find $s\geq 0$ such that $p\cdot\delta_s(v)\in S_0$, which is a contradiction. The second claim clearly follows from the first one.
\end{proof}
Notice that, in view of Proposition \ref{addomestica}, the notion of ruled hypersurface becomes much more simpler in the case of non-characteristic hypersurfaces. Indeed, if $S$ is a closed, non-characteristic ruled hypersurface of class $C^1$ and $p\in S$, then clearly $p\cdot \hhh T_p S\cap S_0=\emptyset$. Therefore a closed non-characteristic hypersurface of class $C^1$ is ruled if and only if it satisfies \eqref{mergestrongruled}.
Now let us discuss some instances of ruled hypersurfaces. We begin with the simplest non-characteristic smooth hypersurface.
\begin{es}[Vertical Hyperplanes]
Let $S$ be a vertical hyperplane of the form
\begin{equation*}
    S=\{p\in\hn\,:\,\langle(\bar x,\bar y),(\bar a,\bar b)\rangle=c\}
\end{equation*}
for some $0\neq (\bar a,\bar b)\in \Ru^{2n}$ and $c\in\Ru$. Without loss of generality, we assume that $a_1\neq 0$. It is easy to see that 
\begin{equation*}
    T_p S=\spann\{(a_2,-a_1,0,\ldots,0),(a_3,0,-a_1,0,\ldots,0),\ldots,(b_n,0,\ldots,0,-a_1,0),T\}
\end{equation*}
for any $p\in S$. Notice that $S_0=\emptyset$. We show that $S$ is ruled. Indeed, noticing that $T\in T_p S$ for any $p\in S$, it follows that
\begin{equation*}
    \hhh T_pS=\spann\{Z_2|_p,\ldots,Z_n|_p,W_1|_p,\ldots,W_n|_p\}
\end{equation*}
for any $p\in S$, where
\begin{equation}\label{baseovunque}
    Z_i=a_i X_1-a_1X_i\qquad\text{and}\qquad W_j=b_j X_1-a_1Y_j
\end{equation}
for any $i=2,\ldots,n$ and $j=1,\ldots,n$. Let now $p=(\bar x,\bar y,t)\in S$, and let $w=(\bar x',\bar y',0)\in \hhh T_pS$. Then there exists $\alpha_j,\beta_j\in\Ru$ such that
\begin{equation*}
    w=\left(\sum_{j=2}^n\alpha_j a_j+\sum_{j=1}^n\beta_j b_j,-\alpha_2 a_1,\ldots,-\beta_n a_1,0\right).
\end{equation*}
We conclude noticing that 
\begin{equation*}
    \langle(\bar x',\bar y'),(\bar a,\bar b)\rangle=a_1\sum_{j=2}^n\alpha_ja_j+a_1\sum_{j=1}^n\beta_j b_j-\sum_{j=2}^n\alpha_ja_1a_j-\sum_{j=1}^n\beta_ja_1b_j=0.
\end{equation*}
\end{es}
Next we consider an instance in the characteristic case.
\begin{es}[Horizontal Hyperplane]
    Let $S$ be the horizontal hyperplane $\hhh_0$. Notice that
    \begin{equation*}
        T_p S=\spann\left\{\frac{\partial}{\partial x_1},\ldots,\frac{\partial}{\partial y_n}\right\}=\spann\{X_1-y_1 T,\ldots,X_n-y_nT,Y_1+x_1T,\ldots,Y_n+x_n T\}
    \end{equation*}
    for any $p\in S$. This in particular implies that $S_0=\{0\}$. Therefore, let $p=(\bar x,\bar y,t)\neq 0$, and assume without loss of generality that $y_1\neq 0$. This implies that
    \begin{equation*}
        \hhh T_pS=\spann\{y_2X_1-y_1X_2,\ldots,y_n X_1-y_1X_n,x_1X_1+y_1Y_1,\ldots,x_nX_1+y_1Y_n\}.
    \end{equation*}
    Therefore, let $w=(z,0)\in \hhh T_p S$, and let $\alpha_j,\beta_j\in\Ru$ be such that
    \begin{equation*}
        z=\left(\sum_{j=2}^n\alpha_jy_j+\sum_{j=1}^n\beta_jx_j,-\alpha_2y_1,\ldots,-\alpha_n y_1,\beta_1y_1,\ldots,\beta_ny_1\right).
    \end{equation*}
    Hence it follows that 
    \begin{equation*}
        Q((\bar x,\bar y),z)=y_1\sum_{j=2}^n\alpha_jy_j+y_1\sum_{j=1}^n\beta_jx_j-\sum_{j=2}^n\alpha_jy_1y_j-\sum_{j=1}^n\beta_jy_1x_j=0.
    \end{equation*}
\end{es}
With the next couple of propositions we show that the class of ruled $C^1$-hypersurfaces is closed under the action of left translations and intrinsic dilations.
\begin{prop}\label{traslaruled}
    Let $S$ be a ruled hypersurface of class $C^1$.
    Then $\tau_q(S)$ is ruled for any $q\in\hn$.
\end{prop} 
\begin{proof}
        Fix $q=(\bar x^q,\bar y^q,t)\in\hn$, define $\tilde S:=\tau_q (S)$ and, given a point $\tilde p\in\tilde S\setminus \tilde S_0$, let $p\in S$ be such that $\tilde p=\tau_q (p)$. Being $\tau_q:S\longrightarrow \tilde S$ a diffeomorphism, then $d\tau_q|_p:T_p S\longrightarrow T_{\tilde{p}}\tilde S$ is an isomorphism. Therefore we have that 
    \begin{equation*}
        d\tau_q|_p(T_pS)=T_{\tilde{p}}\tilde S.
    \end{equation*}
    Moreover, by definition of $\hhh$, it is also the case that
    \begin{equation*}
        d\tau_q|_p(\hhh_p)=\hhh_{\tilde p}.
    \end{equation*}
    Hence we infer that 
    \begin{equation*}
        d\tau_q|_p(\hhh T_p S)=d\tau_q|_p(\hhh_p\cap T_pS)=d\tau_q|_p(\hhh_p)\cap d\tau_q|_p(T_pS)=\hhh_{\tilde{p}}\cap T_{\tilde{p}}\tilde S=\hhh T_{\tilde{p}}\tilde S.
    \end{equation*}
    In particular, notice that $p\in S\setminus S_0$. 
Let $w\in\tilde p\cdot \hhh T_{\tilde p}S$ and assume that there exists $s\geq 0$ maximal with the property that $\tilde p\cdot\delta_\tau(w)\in \tilde S$ for any $\tau\in[0,s].$ We claim that $\tilde p\cdot\delta_s(w)\in \tilde S_0$.
 Let $v=(\bar a,\bar b,0)\in \hhh T_pS$ be such that $d\tau_q|_p(v)=w$.
 By the left-invariance of the horizontal distribution, it follows that $w=(\bar a,\bar b,0)$. Therefore $s$ is maximal with the property that $p\cdot\delta_\tau(v)\in S$ for any $\tau\in[0,s].$ Hence $ p\cdot\delta_s(v)\in S_0$, and so, since 
 \begin{equation*}
     \tilde p\cdot\delta_s(w)=\tilde p\cdot (s\bar a,s\bar b,0)=q\cdot p\cdot(s\bar a,\bar b,0)=q\cdot(p\cdot\delta_s(v))
 \end{equation*}
 and observing that $\tau_q(S_0)=\tilde S_0$, we conclude that $\tilde p\cdot\delta_s(w)\in\tilde S_0$.
\end{proof}
\begin{prop}\label{dilaruled}
        Let $S$ be a ruled hypersurface.
    Then $\delta_\lambda(S)$ is ruled for any $\lambda>0$.
\end{prop}
\begin{proof}
    Fix $\lambda>0$, define $\tilde S:=\delta_\lambda(S)$ and, given a point $\tilde p\in\tilde S\setminus \tilde S_0$, let $p=(\bar x,\bar y,t)\in S$ be such that $\tilde p=\delta_\lambda (p)$. Arguing as in the proof of Proposition \ref{traslaruled}, we get that 
    \begin{equation}\label{tanor}
        d\delta_\lambda|_p(\hhh T_pS)=\hhh T_{\tilde p}\tilde S.
    \end{equation}
    Therefore, again, $p\in S\setminus S_0$.
    Let $w\in\tilde p\cdot \hhh T_{\tilde p}S$ and assume that there exists $s\geq 0$ maximal with the property that $\tilde p\cdot\delta_\tau(w)\in \tilde S$ for any $\tau\in[0,s].$ We claim that $\tilde p\cdot\delta_s(w)\in \tilde S_0$.
    Let $v=(\bar a,\bar b,0)\in \hhh T_p S$ be such that $d\delta_\lambda|_p(v)=w$. We claim that that $w=\delta_\lambda(v)$.
    Indeed, recalling that the Jacobian matrix of $\delta_\lambda$ is a diagonal matrix with diagonal $(\lambda,\ldots,\lambda,\lambda^2)$, then 
 \begin{equation}\label{lambda}
 \begin{split}
       w(f)(q)&=\sum_{j=1}^na_j\frac{\partial (f\circ\delta_\lambda)}{\partial x_j}(p)+\sum_{j=1}^nb_j\frac{\partial (f\circ\delta_\lambda)}{\partial x_j}(p)+\sum_{j=1}^n\left(a_jy_j-b_jx_j\right)T(f\circ\delta_\lambda)(p)\\
       &=\sum_{j=1}^n\lambda a_j\frac{\partial f}{\partial x_j}(\tilde p)+\sum_{j=1}^n\lambda b_j\frac{\partial f}{\partial x_j}(\tilde p)+\sum_{j=1}^n((\lambda a_j)(\lambda y_j)-(\lambda b_j)(\lambda x_j))Tf(\tilde p).       
 \end{split}
 \end{equation}  
    The conclusion then follows as in the previous proof, just noticing that 
    \begin{equation*}
    \delta_{\lambda}(p\cdot\delta_\tau(v))=\delta_\lambda(p)\cdot\delta_\lambda(\delta_\tau(v))=\tilde p\cdot\delta_{\lambda\tau}(v)=\tilde p\cdot\delta_\tau(\delta_\lambda(v))=\tilde p\cdot\delta_\tau(w)
    \end{equation*}
for any $\tau\in\Ru$, and that $\delta_{\lambda}(S_0)=\tilde S_0$.
\end{proof}
In view of Proposition \ref{traslaruled}, we can enlarge the class of examples of ruled hypersurfaces.
\begin{es}[Non-Vertical Hyperplanes]
    We already know that $\hhh_0$ is a characteristic ruled smooth hypersurface. 
    For any fixed $q=(\bar x_q,\bar y_q,t_q)\in\hn$, we know from Proposition \ref{traslaruled} that $\tau_q(\hhh_0)$ is a characteristic ruled smooth hypersurface. 
    Moreover, an easy computation shows that
    \begin{equation*}
        \tau_q(\hhh_0)=\{(\bar x,\bar y,t)\in\hn\,:\,\langle(\bar a,\bar b),(\bar x,\bar y)\rangle+t+d=0\},
    \end{equation*}
    where $(\bar a,\bar b)=(-\bar y_q,\bar x_q)$ and $d=-t_q$. Finally, notice that any hyperplane which is not vertical can be obtained as left-translation of the horizontal hyperplane $\hhh_0$. Hence we conclude that every hyperplane of $\hn$ is ruled, and it is non-characteristic if and only if it is vertical. Finally, notice that we cannot exploit Proposition \ref{dilaruled} to obtain more ruled hypersurfaces, since dilations of hyperplanes are hyperplanes.
\end{es}
To conclude this section, we show that the class of ruled hypersurfaces is closed under the action of the so-called \emph{pseudohermitian transformations} of $\hn$. To introduce this notion, we define the map $\mathcal J:\hn\scu\hn$ by 
\begin{equation*}
    \mathcal J(\bar x,\bar y,t):=(-\bar y,\bar x,t)
\end{equation*}
for any $p=(\bar x,\bar y,t)\in\hn$. The map $\mathcal J$ is a global diffeomorphism which preserves the horizontal distribution, related to the CR structure $J$ by
\begin{equation*}
    d \mathcal J|_{\hhh}=J|_\hhh.
\end{equation*}
A global diffeomorphism $\varphi:\hn\scu \hn$ is said to be a pseudohermitian transformation of $\hn$ if it preserves the horizontal distribution and it commutes with $\mathcal J$, that is
\begin{equation*}
    d\varphi(\hhh)\subseteq \hhh\qquad\text{and}\qquad \varphi\circ \mathcal J=\mathcal J\circ\varphi.
\end{equation*}
Let us begin by considering a special subclass of pseudohermitian transformations. To this aim, let us define the map $\varphi_R:\hn\scu\hn$ by 
\begin{equation}\label{rotaspec}
    \varphi_R(\bar x,\bar y,t):=(R(\bar x,\bar y),t),
\end{equation}
where $R$ is an orthogonal matrix of the form
\begin{equation*}
    R=    \left[
    \begin{array}{cc}
	  A & B \\
	-B & A 
\end{array}
\right],
\end{equation*}
where $A$ and $B$ are real-valued $n\times n$ matrices.
\begin{prop}\label{rotaruled}
    Let $\varphi_R$ be as in \eqref{rotaspec}. Then $\varphi_R$ is a pseudohermitian transformation. Moreover, it holds that
    \begin{equation*}
        d\varphi_R|_p(\bar a,\bar b,0)=(R(\bar a,\bar b),0)
    \end{equation*}
    for any $p\in \hn$ and any $(\bar a,\bar b,0)\in H_p$.
\end{prop}
\begin{proof}
    Let $p=(\bar x,\bar y, t)$ and $(\bar a,\bar b,0)$ as in the statement, and let $\tilde p:=\varphi_R(p)=(\bar{\tilde x},\bar{\tilde y},t)$. We first claim that
    \begin{equation*}
d\varphi_R|_p(X_j|_p)=\sum_{k=1}^n\left(R_{kj}X_k|_{\tilde p}+R_{(n+k)j}Y_k|_{\tilde p}\right)
    \end{equation*}
    and
    \begin{equation*}
d\varphi_R|_p(Y_j|_p)=\sum_{k=1}^n\left(R_{k(n+j)}X_k|_{\tilde p}+R_{n+k)(n+j)}Y_k|_{\tilde p}\right)
    \end{equation*}
for any $j=1,\ldots,n$. Indeed, let $\psi$ be a $C^1$ function defined in a neighborhood of $\tilde p$. Let us recall that, since $(\bar{\tilde x},\bar{\tilde y})=R(\bar x,\bar y)$ and $R$ is orthogonal, then $(\bar x,\bar y)=R^T(\bar{\tilde x},\bar{\tilde y})$, which means, recalling also the special block shape of $R$, that 
\begin{equation*}
    -x_j=\sum_{k=1}^n\left(-R_{kj}\tilde x_k-R_{(n+k)j}\tilde y_k\right)
    =\sum_{k=1}^n\left(-R_{(n+k)(n+j)}\tilde x_k+R_{k(n+j)}\tilde y_k\right)
\end{equation*}
and
\begin{equation*}
    y_j=\sum_{k=1}^n\left(R_{k(n+j)}\tilde x_k+R_{(n+k)(n+j)}\tilde y_k\right)=\sum_{k=1}^n\left(-R_{(n+k)j}\tilde x_k+R_{kj}\tilde y_k\right).
\end{equation*}
for any $j=1,\ldots,n$. Then it holds that
\begin{equation*}
    \begin{split}
        d\varphi_R|_p&(X_j|_p)(\psi)(\tilde p)= X_j|_p(\psi\circ\varphi_R)(p)\\
        &=\frac{\partial}{\partial x_j}(\psi\circ\varphi_R)(p)+y_j T(\psi\circ\varphi_R)(p)\\
&=\sum_{k=1}^n\left(R_{kj}\frac{\partial\psi}{\partial x_k}(\tilde p)+R_{(n+k)j}\frac{\partial\psi}{\partial y_k}(\tilde p) \right)+y_j T(\psi)(\tilde p)\\
&=\sum_{k=1}^n\left(R_{kj}\left(\frac{\partial\psi}{\partial x_k}(\tilde p)+\tilde y_k T(\psi(\tilde p)\right)+R_{(n+k)j}\left(\frac{\partial\psi}{\partial y_k}(\tilde p) -\tilde x_k T(\psi)(\tilde p)\right)\right)\\
& =\sum_{k=1}^n\left(R_{kj}X_k|_{\tilde p}(\psi)(\tilde p)+R_{(n+k)j}Y_k|_{\tilde p}(\psi)(\tilde p)\right)
    \end{split}
\end{equation*}
and, similarly, 
\begin{equation*}
    d\varphi_R|_p(Y_j|_p)(\psi)(\tilde p)=\sum_{k=1}^n\left(R_{k(n+j)}X_k|_{\tilde p}(\psi)(\tilde p)+R_{(n+k)(n+j)}Y_k|_{\tilde p}(\psi)(\tilde p)\right)
\end{equation*}
for any $j=1,\ldots, n$. Hence we conclude that 
\begin{equation*}
\begin{split}
    d\varphi_R|_p&(\bar a,\bar b,0)=\sum_{j=1}^n\left(a_jd\varphi_R|_p(X_j|_p)+b_jd\varphi_R|_p(Y_j|_p)\right)\\
&=\sum_{j,k=1}^n\left(a_j\left(R_{kj}X_k|_{\tilde p}+R_{(n+k)j}Y_k|_{\tilde p}\right)+b_j\left(R_{k(n+j)}X_k|_{\tilde p}+R_{n+k)(n+j)}Y_k|_{\tilde p}\right)\right)\\
&=\sum_{k=1}^n\left(\sum_{j=1}^n\left(R_{kj}a_j+R_{k(n+j)}b_j\right)X_k|_{\tilde p}+\sum_{j=1}^n\left(R_{(n+k)j}a_j+R_{(n+k)(n+j)}b_j\right)Y_k|_{\tilde p}\right)\\
&=(R(\bar a,\bar b),0).
\end{split}
\end{equation*}
\end{proof}
As a consequence of the previous result, it is easy to see that the class of ruled hypersurfaces is closed under the action of maps of the form \eqref{rotaspec}.
\begin{prop}\label{rotaruleddue}
    Let $S$ be a ruled hypersurface. Then $\varphi_R(S)$ is ruled for any $\varphi_R$ as in \eqref{rotaspec}.
\end{prop}
\begin{proof}
    The proof of this result, with the help of Proposition \ref{rotaruled}, follows as the proof of Proposition \ref{traslaruled} and Proposition \ref{dilaruled}, noticing that $\varphi_R(S_0)=(\varphi_R(S))_0$ and that, for a given $p=(z,t)\in S\setminus S_0$, $(v,0)\in \hhh T_pS$ and $s\in\Ru$, it holds that
    \begin{equation*}
   \begin{split}
       \varphi_R(p\cdot \delta_s(v,0))&=\varphi_R(z+sv,t+Q(z,sv))\\
       &=(R(z+sv),t+sQ(z,v))\\
       &=(Rz+sRv,t+sQ(Rz,Rv))\\
       &=(Rz,t)\cdot (sRv,0))\\
       &=\varphi_R(p)\cdot\delta_s(Rv,0).
   \end{split}   
    \end{equation*}
\end{proof}
As a corollary of Proposition \ref{rotaruled}, we can conclude our initial statement.
\begin{thm}\label{pseudoruled}
    If $S$ is ruled, then $\varphi(S)$ is ruled for any pseudohermitian transformation $\varphi$.
\end{thm}
\begin{proof}
    It follows combining Proposition \ref{traslaruled}, Proposition \ref{rotaruleddue} and \cite[Theorem 4.1]{CL}.
\end{proof}
\section{Ruled hypersurfaces with countable characteristic set}\label{ruledsection}
The aim of this section is to characterise ruled hypersurfaces of $\hn$ with countable characteristic set, when $n\geq 2$. In the first Heisenberg group $\mathbb H^1$ there are examples of ruled, non-characteristic, smooth surfaces which are not vertical planes. As an instance, let us consider the surface $S$ parametrized by the map $\varphi:\Ru^2\longrightarrow\mathbb H^1$ defined by
\begin{equation*}
    \varphi(t,\theta):=(t\cos{\theta},t\sin{\theta},\theta).
\end{equation*}
Notice that $\varphi$ is smooth and injective. Moreover, 
\begin{equation*}
    \frac{\partial\varphi}{\partial t}(t,\theta)=\cos{\theta}\frac{\partial}{\partial x}+\sin{\theta}\frac{\partial}{\partial y}=\cos{\theta} X|_{\varphi(t,\theta)}+\sin{\theta}Y|_{\varphi(t,\theta)}
\end{equation*}
and
\begin{equation*}
     \frac{\partial\varphi}{\partial \theta}(t,\theta)=-t\sin{\theta}\frac{\partial}{\partial x}+t\cos{\theta}\frac{\partial}{\partial y}+T=-t\sin{\theta} X|_{\varphi(t,\theta)}+t\cos{\theta}Y|_{\varphi(t,\theta)}+(1+t^2)T.
\end{equation*}
This implies that $S$ is a smooth, non-characteristic surface, and moreover
\begin{equation*}
    \hhh T_{\varphi(t,\theta)} S=\spann\left\{\frac{\partial\varphi}{\partial t}(t,\theta)\right\}
\end{equation*}
for any $(t,\theta)\in \Ru^2$. Finally, for given $t,\theta,s\in\Ru$, it holds that
\begin{equation*}
    (t\cos{\theta},t\sin{\theta},\theta)\cdot(s\cos{\theta},s\sin{\theta},0)=((t+s)\cos{\theta},(t+s)\sin{\theta},\theta)\in S,
\end{equation*}
and so $S$ is ruled. 
However, the situation in higher dimensional Heisenberg groups is quite different, and the ruling condition turns out to be more restrictive. Indeed, we are going to prove that the only closed, ruled hypersurfaces with countable characteristic set in $\hn$, with $n\geq 2$, are hyperplanes.
 To this aim, we already know that vertical hyperplanes are non-characteristic and ruled, and that every non-vertical hyperplane
\begin{equation*}
    P:=\left\{(\bar x,\bar y,t)\in\hn\,:\,\sum_{j=1}^na_jx_j+\sum_{j=1}^nb_jy_j+ct+d=0\right\},
\end{equation*}
where clearly $c\neq 0$, is ruled and satisfies
\begin{equation*}
    P_0=\left\{\left(\frac{b_1}{c},\ldots,\frac{b_n}{c},-\frac{a_1}{c},\ldots,-\frac{a_n}{c},-\frac{d}{c}\right)\right\}.
\end{equation*}

Before proving \Cref{rulcar} we establish some preliminary results.
\begin{proposition}\label{lemmutile}
    Let $S$ be a hypersurface of class $C^1$. Assume that $S$ is closed and without boundary. Assume that $S$ is ruled and that $S_0$ is countable.
    Then 
    \begin{equation*}
        p\cdot \hhh T_p S\subseteq S
    \end{equation*}
    for any $p\in S\setminus S_0$.
\end{proposition}
\begin{proof}
    Let $S$ and $p$ as in the statement. Assume by contradiction that there exists $q\in p\cdot \hhh T_p S\setminus S$. Combining \Cref{addomestica} with the fact that $S_0$ is countable and that $S$ is ruled, it is easy to construct a sequence $(q_h)_h\subseteq S$ converging to $q$ as $h\to\infty$. Being $S$ closed, then $q\in S$, a contradiction.
\end{proof}
\begin{proposition}\label{analogodi3punto2}
  Let $S$ be a hypersurface of class $C^1$. Assume that $S$ is closed and without boundary. Assume that $S$ is ruled and that $S_0$ is countable. Assume that $0\in S\setminus S_0$. Then
      \begin{equation*}
        S\cap \hhh_0=\hhh T_0 S.
    \end{equation*}
\end{proposition}
\begin{proof}
      First, since $0\in S\setminus S_0$ and in view of \Cref{lemmutile}, then $\hhh T_0 S\subseteq S\cap \hhh_0$. Assume by contradiction that there exists $q=(z_q,0)\in (S\cap \hhh_0)\setminus \hhh T_0 S$. If $S$ is tangent to $\hhh_0$ at $q$, then $q\in S\setminus S_0$. Otherwise, since $\hhh T_0 S$ is closed and $S_0$ is countable, it is possible to find another point in $(S\setminus S_0)\setminus \hhh T_0 S$. In the end, we can assume that $q\in S\setminus S_0$. Again, thanks to \Cref{lemmutile}, $q\cdot \hhh T_q S\subseteq S$, and so $q\cdot \hhh T_q S\cap \hhh_0\subseteq S\cap \hhh_0$.  Note that both $\hhh T_0 S$ and $q\cdot \hhh T_q S\cap \hhh_0$ are affine subspaces of $\hhh_0$. Moreover, $\dim(\hhh T_0 S)=2n-1$ and $\dim(q\cdot \hhh T_q S\cap \hhh_0)\geq 2n-2$. Therefore we conclude that 
    \begin{equation*}
            \dim (\hhh T_0 S\cap (q\cdot \hhh T_q S\cap \hhh_0))\geq\dim (\hhh T_0 S)+\dim (q\cdot \hhh T_q S\cap \hhh_0)-2n=2n-3\geq 1,
    \end{equation*}
    since $n\geq 2$. Therefore $(\hhh T_0 S)\cap (q\cdot \hhh T_q S\cap H_0)$ contains a one-dimensional affine subspace of $H_0$. In particular, being $S_0$ countable, there exists $p=(z_p,0)\in \hhh T_0 S\cap (q\cdot \hhh T_q S\cap \hhh_0)\cap (S\setminus S_0)$. Let $v\in \hhh T_p S$ be such that $p\cdot tv=q$ for some $t\in \rr$, and let $\gamma_p(t):=(tz_p,0)$. Notice that, by construction, then $\gamma_p (t)\in S$ for any $t\in\rr$. Moreover, $\Dot\gamma_p(1)=(z_p,0)\in \hhh_p$, and so $w:=(z_p,0)\in \hhh T_p S$. Again, since $p\in S\setminus S_0$ and in view of \Cref{lemmutile}, then $p\cdot \hhh T_p S\subseteq S$. Therefore, in particular, it holds that 
    \begin{equation*}
        p\cdot(\alpha v+\beta w)\in S
    \end{equation*}
    for any $\alpha,\beta\in\rr$. Hence, if we let $\gamma_q(t):=(tz_q,0)$, we conclude that $\gamma(t)\in S\cap \hhh_0$ for any $t\in\rr,$ and so $\Dot\gamma_q(0)=(z_q,0)\in T_0 S$. Since clearly $(z_q,0)\in \hhh_0$, then $q\in \hhh T_0 S$, which is a contradiction.
\end{proof}

\begin{proposition}\label{tgraf}
    Let $S$ be a hypersurface of class $C^1$. Assume that $S$ is closed and without boundary. Assume that $S$ is ruled and that $S_0$ is countable. Then either $S$ is a $t$-graph or $S$ is a vertical hyperplane.
\end{proposition}
\begin{proof}
    If $S$ is a $t$-graph we are done. If $S$ is not a $t$-graph, being $S_0$ countable, there exists $p\in S\setminus S_0$ such that $T|_p\in T_p S$. Up to a left-translation, recalling \Cref{traslaruled}, we assume that $p=0$. We show that $S$ is a vertical hyperplane, dividing the proof into some steps.\\
    \textbf{Step 1.}
    Thanks to \Cref{analogodi3punto2}, we know that there exists $0\neq(\bar a,\bar b)\in \rr^{2n}$ such that
    \begin{equation*}
        \hhh T_0 S=\hhh_0\cap S=\{(\bar x,\bar y,0)\in \hn\,:\,\langle (\bar a,\bar b),(\bar x,\bar y)\rangle=0\}.
    \end{equation*}
    We assume without loss of generality that $a_1\neq 0$, and we let $f(\bar x,\bar y):=\langle (\bar a,\bar b),(\bar x,\bar y)\rangle$.
    We claim that 
    \begin{equation*}
        \pi(p\cdot \hhh T_p S)\subseteq \pi (\hhh T_0 S)
    \end{equation*}
    for any $p\in \hhh T_0 S\cap (S\setminus S_0)$, where here and in the following the map $\pi:\hn\longrightarrow\rr^{2n}$ is defined by
\begin{equation*}
    \pi(\bar x,\bar y,t):=(\bar x,\bar y).
\end{equation*}
Assume by contradiction that there exists $p=(z_p,0)\in \hhh T_0 S\cap (S\setminus S_0)$ and $v=(v,0)\in \hhh T_p S$ such that $z_p+v\notin\pi(\hhh T_0 S)$. This is equivalent to say that $f(z_p+v)\neq 0$.
    Let us define $q:=p\cdot v=(z_p+v, Q(z_p,v))$. Since $p\in S\setminus S_0$ and by \Cref{lemmutile}, then $q\in S$. Moreover, $Q(z_p,v)\neq 0$, since otherwise $q\in \hhh T_0 S$ and consequently $f(z_p+v)=0$.
    Moreover, since $z_p\in \hhh T_0 S$, then, letting $\gamma (t):=(tz_p,0)$, it holds that $\gamma(t)\in S$ for any $t\in\rr$, and so $(z_p,0)\in \hhh T_p S$. Hence, since $p\cdot \hhh T_p S\subseteq S$, we conclude in particular that
    \begin{equation*}
       P:=\{(z_p,0)+\alpha(z_p,0)+\beta (v, Q(z_p,v))\,:\,\alpha,\beta\in\rr\}\subseteq S.
    \end{equation*}
    Notice that $P$ is a vector subspace of $\rr^{2n+1}$. Then in particular $0\in P$ and $ (v,Q(z_p,v))\in T_0 S$.
    Therefore, as $T\in T_0 S$, then $(v,0)\in T_0 S$, and so, since $(v,0)\in \hhh_0$, we conclude that $(v,0)\in \hhh T_0 S$. Then $f(v)=0$, and so, as $p\in \hhh T_0 S$, $f(z_p+v)=f(z_p)+f(v)=0$, a contradiction.\\
    \textbf{Step 2.} Let $p=(z_p,0)\in \hhh T_0 S\cap (S\setminus S_0)$. Thanks to Step 1, we know that $\pi(p\cdot \hhh T_p S)\subseteq \pi (\hhh T_0 S)$. Therefore, if $v\in \hhh T_p S$, then $f(z_p+v)=0$. Since $f(z_p)=0,$ we conclude that $f(v)=0$, which implies that
    \begin{equation}\label{htuguali}
        \hhh T_pS=\hhh T_0 S
    \end{equation}
    for any $p\in \hhh T_0 S$. Moreover, an easy computation shows that 
    \begin{equation*}
    \hhh T_0 S=\spann\{Z_2|_0,\ldots,Z_n|_0,W_1|_0,\ldots,W_n|_0\},
\end{equation*}
where \begin{equation}\label{baseovunque}
    Z_i=a_i X_1-a_1X_i\qquad\text{and}\qquad W_j=b_j X_1-a_1Y_j
\end{equation}
for any $i=2,\ldots,n$ and $j=1,\ldots,n$.
 Then \eqref{htuguali} allows to conclude that 
    \begin{equation}\label{htcome}
    \hhh T_pS=\spann\{Z_2|_p,\ldots,Z_n|_p,W_1|_p,\ldots,W_n|_p\}.
\end{equation}
\textbf{Step 3.}
Let us define 
\begin{equation*}
    \mathcal Z:=\{z\in\pi(\hhh T_0 S)\,:\,Q(z,w)=0\text{ for any }w\in\pi (\hhh T_0 S)\}.
\end{equation*}
    Notice that, being $Q$ a bilinear map, then $\mathcal Z$ is a vector subspace of $\pi(\hhh T_0 S)$. We claim that $\dim (\mathcal{Z})\leq 2n-2$. Indeed, assume by contradiction that $\dim (\mathcal{Z})\geq 2n-1$. Then, since $\mathcal Z\subseteq \pi(\hhh T_0 S)$ and $\dim (\pi(\hhh T_0 S))=2n-1$, we conclude that $\mathcal{Z}=\pi (\hhh T_0 S)$. We show that this leads to a contradiction. Assume first that $a_2=\ldots=a_n=b_2=\ldots=b_n=0$, and set $z_1=(0,-1,0\ldots,0)$ and $z_2=(\bar 0,0,1,0,\ldots, 0)$. Then $f(z_1)=f(z_2)=0$ and $Q(z_1,z_2)=1\neq 0$, which implies that $z_1,z_2\notin \mathcal Z$. If it is not the case that $a_2=\ldots=a_n=b_2=\ldots=b_n=0$, then assume without loss of generality that $a_2\neq 0$. Let $z_1=(-a_2,a_1,0,\ldots,0)$ and $z_2=(-b_1,0,\ldots, 0,a_1,0,\ldots,0)$. Then $f(z_1)=f(z_2)=0$ and $Q(z_1,z_2)=a_1a_2\neq 0$, which implies again that $z_1,z_2\notin\mathcal Z$. Therefore we conclude that $\dim(\mathcal Z)\leq 2n-2$, and so in particular
\begin{equation}\label{chiusura}
    \overline{\pi(\hhh T_0 S)\setminus \mathcal K}=\pi(\hhh T_0 S).
\end{equation}
\textbf{Step 4.}
We claim that for any $q=(z_q,t_q)=(x_1^q,\ldots,x_n^q,y_1^q,\ldots,y_n^q,t_q)$ such that $z_q\in \pi(\hhh T_0 S)\setminus\mathcal Z$ there exists $p=(z_p,0)=(x_1^p,\ldots,x_n^p,y_1^p,\ldots,y_n^p,0)\in \hhh T_0 S\cap (S\setminus S_0)$ and $v\in \hhh T_p S$ such that
\begin{equation}\label{raggiungere}
    q=p\cdot v.
\end{equation}
Indeed, let $q$ as above, and let $p\in \hhh T_0 S\cap (S\setminus S_0)$ and $v\in \hhh T_p S$ to be chosen later. In view of \eqref{htcome}, we can express $v$ as
\begin{equation*}
    v=\sum_{j=2}^n\alpha_j Z_j|_p+\sum_{j=1}^n\beta_j W_j|_p=\left(\sum_{j=2}^n\alpha_ja_j+\sum_{j=1}^n\beta_jb_j\right)X_1|_p-\sum_{j=2}^n\alpha_j a_1X_j|_p-\sum_{j=1}^n\beta_ja_1Y_j|_p.
\end{equation*}
for some $\alpha_2,\ldots,\alpha_n,\beta_1,\ldots,\beta_n\in\rr$. Therefore, we infer that
\begin{equation*}
        p\cdot v=\left(x_1^p+\sum_{j=2}^n\alpha_ja_j+\sum_{j=1}^n\beta_jb_j,x_2^p-\alpha_2a_1,\ldots,y_n^p-\beta_na_1, Q(z_p,v)\right).
\end{equation*}
Let us choose
\begin{equation*}
    \alpha_i=\frac{x_i^p-x_i^q}{a_1}\qquad\textit{and}\qquad \beta_j=\frac{y_j^p-y_j^q}{a_1}
\end{equation*}
for any $i=2,\ldots,n$ and any $j=1,\ldots,n$. This choice implies that $(p\cdot v)_i=x_i^q$ and $(p\cdot v)_j=y_{j-n}^q$
for any $i=2,\ldots,n$ and any $j=n+1,\ldots,2n$. Moreover, since $f(z_p)=f(z_q)=0$, it holds that 
\begin{equation*}
        (p\cdot v)_1= x_1^p+\sum_{j=2}^n\alpha_ja_j+\sum_{j=1}^n\beta_jb_j=\frac{1}{a_1}\left(\sum_{j=1}^n(a_jx_j^p+b_jy_j^p)-\sum_{j=2}^ma_jx_j^q+\sum_{j=1}^nb_jy_j^q\right)
        =x_1^q.
\end{equation*}
Finally, notice that
\begin{equation*}
    \begin{split}
Q(z_p,v)&=\left(\sum_{j=2}^n\alpha_ja_j+\sum_{j=1}^n\beta_jb_j\right)y_1^p-\sum_{j=2}^n\alpha_ja_1y_j^p+\sum_{j=1}^n\beta_ja_1x_j^p\\
&=\frac{1}{a_1}\Big(\sum_{j=2}^na_jx_j^py_1^p-\sum_{j=2}^na_jx_j^qy_1^p+\sum_{j=1}^nb_jy_j^py_1^p-\sum_{j=1}^nb_jy_j^qy_1^p\\
&-\sum_{j=2}^na_1x_j^py_j^p+\sum_{j=2}^na_1 x_j^qy_j^p+a_1x_1^py_1^p+\sum_{j=2}^na_1x_j^py_j^p-\sum_{J=1}^na_1x_j^py_j^q\Big)\\
&= \frac{1}{a_1}\Big(-\sum_{j=1}^na_jx_j^qy_1^p-\sum_{j=1}^nb_jy_j^qy_1^p+\sum_{j=1}^na_1 x_j^qy_j^p-\sum_{j=1}^na_1x_j^py_j^q\Big)\\
&=Q(z_p,z_q),
    \end{split}
\end{equation*}
where in the third equality we exploited the fact that $f(z_p)=0$, while the fourth equality follows from $f(z_q)=0$. Since we assumed $z_q\notin\mathcal Z$, then there exists uncountably many $w\in \pi(\hhh T_0 S)$ such that $Q(w,z_q)\neq 0.$ Therefore, since $S_0$ is countable, it is possible to choose $w\in \pi(\hhh T_0 S)$ such that, setting 
\begin{equation*}
    z_p=\frac{t_q}{Q(w,z_q)}w,
\end{equation*}
then $p\in (S\setminus S_0)$. We conclude that $p\in \hhh T_0 S\cap (S\setminus S_0)$ and $Q(z_p,z_q)=t_q$.\\
\textbf{Step 5.} We are now able to conclude. Indeed, thanks to \eqref{raggiungere} we infer that 
\begin{equation*}
    \pi(HT_0S\setminus\mathcal K)\times\Ru\subseteq S.
\end{equation*}
But then, being $S$ closed and recalling \eqref{chiusura}, we conclude that
\begin{equation*}
    \pi(HT_0S)\times \Ru=\overline{\pi(HT_0S)\setminus\mathcal K}\times \Ru=\overline{\pi(HT_0S)\times \Ru}\subseteq \overline S=S.
\end{equation*}
Therefore $S$ contains the vertical hyperplane $\pi(HT_0S)\times\Ru$. The thesis then follows in view of the topological assumptions on $S$.
\end{proof}
\begin{proof}[Proof of \Cref{rulcar}]
    Let $S$ be as in the statement. If $S$ is a vertical hyperplane, we are done. If not, in view of \Cref{tgraf}, $S$ is a $t$-graph. Being $S_0$ countable, and recalling \Cref{traslaruled}, up to a left translation we may assume that $0\in S\setminus S_0$ and that $T|_0\notin T_0 S$. Since $0\in S\setminus S_0$, we infer by \Cref{analogodi3punto2} that
\begin{equation*}
    \hhh T_0 S=\hhh_0\cap S=\{(\bar x,\bar y,0)\in\hn\,:\,\langle (\bar x,\bar y),(\bar a,\bar b)\rangle=0\}
\end{equation*}
for some $0\neq (\bar a,\bar b)\in\rr^{2n}$. Again, by \Cref{rotaruled} we may assume that $a_1\neq 0$. Being $S$ an entire $t$-graph, and since $T|_0\notin T_0 S$ and $0\in S\setminus S_0$, there exists $v=(z_v,t_v)\in T_0 S$ such that $f(z_v)\neq 0$ and $t_v\neq 0$. Let us set $c:=-\frac{f(z_v)}{t_v}$. We claim that 
\begin{equation*}
    S=\{(z,t)\in\hn\,:\,f(z)+ct=0\}=:S_c.
\end{equation*}
Indeed, let $p=(z_p,t_p)\in S\setminus S_0.$ If $t_p=0$, then $f(z_p)=0$, and so $p\in S_c$. Assume then $t_p\neq 0$. Then, being $S$ a $t$-graph, we infer that $f(z_p)\neq 0.$.
Let $v_1,\ldots,v_{2n-1}$ be a basis of $\hhh T_p S$. Since $p\in S\setminus S_0$ and thanks to \Cref{lemmutile}, then $p\cdot \hhh T_p S\subseteq S$. 
We claim that there exists $j=1,\ldots,2n-1$ such that $Q(z_p,v_j)\neq 0$. 
    Indeed, assume by contradiction that $  Q(z_p,v_1)=\ldots=Q(z_p,v_{2n-1})=0.$
In this case, recalling that $S$ is ruled, it holds that
\begin{equation}\label{parallelo}
    p\cdot \hhh T_p S=\left\{\left(z_p+\sum_{j=1}^{2n-1}\alpha_jv_j,t_p\right)\,:\,\alpha_1,\ldots,\alpha_{2n-1}\in\rr\right\}\subseteq S.
\end{equation}
We claim that 
\begin{equation}\label{duespan}
    \spann\{(v_1,0),\ldots,(v_{2n-1},0)\}=\spann\{Z_2|_0,\ldots,Z_n|_0,W_1|_0,\ldots,W_n|_0\},
\end{equation}
where $Z_2,\ldots,Z_n,W_1,\ldots,W_n$ are defined as \eqref{baseovunque}. Indeed, if it was not the case, then \eqref{parallelo} would imply the existence of $q=(z_q,t_p)\in S$ such that $z_q\in\pi(\hhh T_0 S)$. But since $t_p\neq 0$ and since $(z_q,0)\in S$, we would contradict the fact that $S$ is a $t$-graph. Notice that \eqref{duespan} implies that
\begin{equation*}
Z_2|_0,\ldots,Z_n|_0,W_1|_0,\ldots,W_n|_0\in H_p
\end{equation*}
and so, observing that 
\begin{equation*}
    Z_j|_0=a_j\frac{\partial}{\partial x_1}-a_1\frac{\partial}{\partial x_j}=a_j X_1|_p-a_1X_j|_p+(a_1y_j-a_j y_1)T
\end{equation*}
for any $j=2,\ldots,n$ and
    \begin{equation*}
    W_j|_0=b_j\frac{\partial}{\partial x_1}-a_1\frac{\partial}{\partial y_j}=b_j X_1|_p-a_1Y_j|_p+(-a_1x_j-b_j y_1)T
\end{equation*}
for any $j=1,\ldots,n$, we conclude that 
\begin{equation*}
    z_p=\frac{y_1}{a_1}(-b_1,\ldots,-b_n,a_1,\ldots,a_n),
\end{equation*}
which implies in particular that $f(z_p)=0$, a contradiction.
In this case, it holds that $p\cdot H T_p S\cap \hhh_0\cap S=p\cdot \hhh T_p S\cap \hhh T_0 S\neq 0$. Since $n\geq 2$, a dimensional argument as in the proof of \Cref{analogodi3punto2} implies that $\dim (p\cdot \hhh T_p S\cap \hhh_0\cap S)\geq 1$. Therefore, being $S_0$ countable, there exists $q=(z_q,0)\in (p\cdot \hhh T_p S)\cap \hhh T_0 S\setminus S_0$. Let then $w\in \hhh T_p S$ be such that
\begin{equation}\label{pq}
    (z_q,0)=(z_p+w,t_p+Q(z_pw)).
\end{equation}
Arguing as in the proof of \Cref{analogodi3punto2}, recalling that $q\in S\setminus S_0$ and \Cref{lemmutile}, we see that
  \begin{equation*}
       P:=\{(z_q,0)+\alpha(z_q,0)+\beta (w, Q(z_q,w))\,:\,\alpha,\beta\in\rr\}\subseteq S,
    \end{equation*}
and so we conclude as above that $(w.Q(z_q,w))\in T_0 S$. This means that there exists $\tilde w\in \pi(\hhh T_0 S)$ and $\alpha\in\rr$ such that
\begin{equation*}
    (w.Q(z_q,w))=(\tilde w+\alpha z_v,\alpha t_v).
\end{equation*} 
Therefore, recalling \eqref{pq}, we get that
\begin{equation*}
     (z_p,t_p)=(z_q,0)-(w,Q(z_p,w))  =\left(z_q-\tilde w-\alpha z_v,-\alpha t_v\right).
\end{equation*}
Therefore, since $z_q,\tilde w\in\pi(\hhh T_0 S)$ we conclude that  
\begin{equation*}
    f(z_p)+ct_p=-\alpha (f(z_v)+ct_v)=0,
\end{equation*}
which implies that $p\in S_c$. Therefore we proved that $S\setminus S_0\subseteq S_c$, and so, being $S_0$ countable and $S$ and $S_c$ closed, we conclude that $S\subseteq S_c$. The thesis then follows by the topological assumptions on $S$.
\end{proof}

\section{Horizontal second fundamental forms on $\hh^n$}\label{secfundsection}
In the current literature, different kinds of second fundamental form are available in the sub-Riemannian setting. We recall that the \emph{horizontal second fundamental form of $S$ at $p\in S\setminus S_0$} (cf. \cite{MR2401420,MR2354992,MR3385193}) is the map $h_p:\hhh T_pS\times \hhh T_pS\longrightarrow \rr$ defined by
\begin{equation*}
    h_p(X,Y)=-\langle\nabla_XY,\vh\rangle=\langle\nabla_X\vh,Y\rangle
\end{equation*}
for any $X,Y\in \hhh T_pS$, the second equality following being $\nabla$ a metric connection. 
We recall that its norm is defined by
\begin{equation*}
    |h_p|^2=\sum_{i,j=1}^{2n-1}h_p(e_i,e_j)^2
\end{equation*}
for any $p\in S$, being $e_1,\ldots,e_{2n-1}$ any orthonormal basis of $\hhh T_pS$. Notice that, in view of \eqref{pseudotorsion}, $h$ may not be symmetric. 
The \emph{horizontal mean curvature $H_p$ of $S$ at $p\in S\setminus S_0$} is defined by
\begin{equation*}
    H_p=\tau(h_p)=\sum_{j=1}^{2n-1}h_p(e_j,e_j),
\end{equation*}
where here and in the following $\tau$ denote the trace operator.
In analogy with the Riemannian case, the horizontal mean curvature coincides with the divergence of the horizontal normal (cf. \cite{MR2354992}), meaning that
\begin{equation*}
   H=\divv_\hh\vh(p)
\end{equation*}
for any $p\in S$. Accordingly, the following characterization of $|h|^2$ holds. 
\begin{proposition}\label{norofh}
    Let $p$ be a non-characteristic point of $S$. Let $\vh$ be any unitary $C^2$ extension of $\vh|_S$. Then 
    \begin{equation*}
        | h_p|^2=\sum_{h,k=1}^{2n}Z_h(\vh_k)Z_k(\vh_h)+4(n-1)(T d^\hh)^2.
    \end{equation*}
Moreover, if $\vh|_S$ is extended as in \eqref{normcondist}, then
\begin{equation*}
    | h_p|^2=\sum_{i,j=1}^{2n}\left(Z_i\vh_j(p)\right)^2-4(T d^\hh)^2.
\end{equation*}
    \end{proposition}
    \begin{proof}
    First, we show that the quantity $\sum_{h,k=1}^{2n}Z_h(\vh_k)Z_k(\vh_h)$ does not depend on the chosen unitary $C^2$ extension of $\vh|_S$. Indeed, in view of \eqref{propvh1}, we have that
    \begin{equation*}
    \begin{split}
\sum_{h,k=1}^{2n}Z_h(\vh_k)Z_k(\vh_h)&=\sum_{h,k=1}^{2n}\delta_h(\vh_k)Z_k(\vh_h)+\sum_{h,k=1}^{2n}\vh_h\langle \nabla_\hh\vh_k,\vh\rangle Z_k(\vh_h)\\
&=\sum_{h,k=1}^{2n}\delta_h(\vh_k)\delta_k(\vh_h)+\sum_{h,k=1}^{2n}\delta_h(\vh_k)\vh_k\langle\nabla_\hh\vh_h,\vh\rangle\\
&=\sum_{h,k=1}^{2n}\delta_h(\vh_k)\delta_k(\vh_h)+\sum_{h,k=1}^{2n}Z_h(\vh_k)\vh_k\langle\nabla_\hh\vh_h,\vh\rangle-\left(\sum_{h,k=1}^{2n}\vh_h\langle\nabla_\hh\vh_h,\vh\rangle\right)^2\\
&=\sum_{h,k=1}^{2n}\delta_h(\vh_k)\delta_k(\vh_h).
    \end{split}
    \end{equation*}
The claim then follows recalling that the horizontal tangential derivatives do not depend on the chosen extension. Let us extend $\vh$ as in \eqref{normcondist}. Let $e_1,\ldots,e_{2n-1}$ be an orthonormal basis of $\hhh T_pS$. For any $i=1,\ldots,2n-1$, we let $a_i^1,\ldots,a_i^{2n}$ be such that
        \begin{equation*}
            e_i=\sum_{j=1}^{2n}a_i^jZ_j.
        \end{equation*}
        Then, by construction,
        \begin{equation}\label{coeffequations}
            \sum_{k=1}^{2n}a_i^ka_j^k=\delta_{ij},\qquad\sum_{k=1}^{2n}a_i^k\vh_k=0\qquad\text{and}\qquad\sum_{l=1}^{2n-1}e_k^le_k^m=\delta_{km}-\vh_k\vh_m
        \end{equation}
        for any $i,j=1,\ldots,2n-1$ and any $l,m=1,\ldots,2n$. Hence, recalling \eqref{propvh1} and \eqref{propvh5},
        \begin{equation*}
            \begin{split}
    |h_p|^2             &=\sum_{i,j=1}^{2n-1}\sum_{h,k,l,m=1}^{2n}a_i^hZ_h(\vh_k)a_j^ka_i^lZ_l(\vh_m)a_j^m\\
&=\sum_{h,k,l,m=1}^{2n}Z_h(\vh_k)Z_l(\vh_m)(\delta_{hl}-\vh_h\vh_l)(\delta_{km}-\vh_k\vh_m)\\
&=\sum_{h,k=1}^{2n}\left(Z_h(\vh_k)\right)^2-\sum_{k=1}^{2n}\left(\sum_{h=1}^{2n}Z_h(\vh_k)\vh_h\right)^2\\
&=\sum_{i,j=1}^{2n}\left(Z_i\vh_j(p)\right)^2-4(T d^\hh)^2.
            \end{split}
        \end{equation*}
        To prove the second identity, notice that

    \begin{equation*}
    \begin{split}
        \sum_{h,k=1}^{2n}Z_h(\vh_k)Z_k(\vh_h)&
=\sum_{h,k=1}^{2n}\left(Z_h(\vh_k)\right)^2+2Td^\hh\sum_{h,k=1}^nX_h(\vh_{n+k})-2Td^\hh\sum_{h,k=1}^nY_h(\vh_k)\\
        &=\sum_{h,k=1}^{2n}\left(Z_h(\vh_k)\right)^2+2Td^\hh\sum_{k=1}^n[X_k,Y_k]d^\hh\\
        &=\sum_{h,k=1}^{2n}\left(Z_h(\vh_k)\right)^2-4n(Td^\hh)^2
    \end{split}
    \end{equation*}
    \end{proof}

In the first Heisenberg group $\hh^1$, $\hhh TS$ is a one dimensional distribution generated by $J(\vh)$. In particular (cf. \cite{MR2435652}), $h$ completely determines the behavior of $\nabla_{J(\vh)}J(\vh)$, meaning that
\begin{equation*}
    \nabla_{J(\vh)}J(\vh)=-h(J(\vh),J(\vh))\vh.
\end{equation*}
This consideration is a first crucial step in the study of minimal surfaces, since it allows to infer that, when $H=0$, then $S$ is ruled by horizontal lines. A horizontal line is a horizontal curve $\Gamma:I\longrightarrow\hn$ such that 
\begin{equation*}
\langle\Ddot\Gamma(s),Z_j|_{\Gamma(s)}\rangle=0    
\end{equation*}
for any $s\in I$ and any $j=1,\ldots,2n$, where here and in the following $I$ is a domain of $\rr$ containing $0$.
Indeed the following simple characterization holds. 
\begin{proposition}\label{charinhn}
     Let $\Gamma:I\longrightarrow\hh^n$ be a horizontal curve. The following are equivalent. 
    \begin{itemize}
        \item [$(i)$] $\nabla_{\Dot\Gamma} \Dot\Gamma=0$ along $\Gamma$.
        \item [$(ii)$]$\Gamma$ is a horizontal line.
    \end{itemize}
\end{proposition}
\begin{proof}
    Let $A=\sum_{j=1}^{2n}A_jZ_j$ be any $C^2$ extension of $\Dot\Gamma$. $\Gamma$ is a horizontal line if and only if $t\mapsto A_j(\Gamma(t))$ is constant on $I$ for any $j=1,\ldots,2n$.
    Notice that
    \begin{equation*}
        \nabla_A A\big |_{\Gamma(s)}=\sum_{k=1}^{2n}A(A_k)|_{\Gamma(s)} Z_k|_{\Gamma(s)}=\sum_{k=1}^{2n}\Dot\Gamma(A_k)|_{\Gamma(s)} Z_k|_{\Gamma(s)}=\sum_{k=1}^{2n}\frac{d(A_k(\Gamma(t))}{dt}\Big|_{s}Z_k|_{\Gamma(s)}
    \end{equation*}
    for any $s\in I$. The thesis then follows.
\end{proof}
The higher dimensional case is typically more involved, since it is not always true that
\begin{equation*}
    \nabla_X X=-h(X,X)\vh.
\end{equation*}
 Nevertheless, there is a particular situation in which the second fundamental form provides global information. 
\begin{definition}\label{htgdef}
    Let $S$ be a hypersurface of class $C^2$. We say that $S$ is \emph{horizontally totally geodesic} when 
\begin{equation}\label{htg}
    h(X,X)=0
\end{equation}
for any $X\in C^1(S,\hhh TS)$, that is when $h$ is an \emph{alternating} bilinear form.
\end{definition}
Notice that \eqref{htg} is equivalent to require that the symmetric part of $h$ is identically vanishing. Let us denote the latter by $\tilde h$, that is
\begin{equation*}
    \tilde h(X,Y)=\frac{h(X,Y)+h(Y,X)}{2}
\end{equation*}
For any $X,Y\in C^1(S,\hhh TS)$. This kind of symmetric second fundamental form has already been considered, although through different but equivalent definitions, by several authors (cf. e.g. \cite{MR2898770,MR2354992,MR4193432}). It is clear that
\begin{equation*}
    h_p=0\implies \tilde h_p=0
\end{equation*}
for any non-characteristic point $p\in S$, while the converse implication may be false in general.
More precisely, $|h|$ and $|\tilde h|$ can be related in the following way.
\begin{proposition}\label{normoftildeh}
    Let $p$ be a non-characteristic point of $S$. Let $\vh$ be any unitary $C^2$ extension of $\vh|_S$. Then 
    \begin{equation*}
       |\tilde h_p|^2=\sum_{h,k=1}^{2n}Z_h(\vh_k)Z_k(\vh_h)+2(n-1)\left(Td^\hh(p)\right)^2.
    \end{equation*}
    Moreover,
    \begin{equation*}
        |h_p|^2=|\tilde h_p|^2+2(n-1)\left(Td^\hh(p)\right)^2.
    \end{equation*}
    Finally, if $\vh$ is extended as in \eqref{normcondist}, then
    \begin{equation*}
        |\tilde h_p|^2=\sum_{h,k=1}^{2n}\left(Z_h\vh_k(p)\right)^2-2(n+1)\left(Td^\hh(p)\right)^2
    \end{equation*}
    \end{proposition}
    \begin{proof}
        Let $e_1,\ldots,e_{2n-1}$ be as in the proof of \Cref{norofh}. Notice that
\begin{equation*}
    |\tilde h_p|^2=\tau\left(\tilde h_p\cdot\tilde h_p^T\right)=\tau\left(\frac{(h_p+h_p^T)^2}{4}\right)=\frac{1}{2}|h_p|^2+\frac{1}{2}\tau\left(h_p^2\right).
\end{equation*}
Arguing as in the proof of \Cref{norofh},
        \begin{equation*}
            \begin{split}
    \tau\left(h_p^2\right)&=\sum_{i,j=1}^{2n-1}\sum_{h,k,l,m=1}^{2n}a_i^hZ_h(\vh_k)a_j^ka_j^lZ_l(\vh_m)a_i^m\\
              &=\sum_{h,k,l,m=1}^{2n}Z_h(\vh_k)Z_l(\vh_m)\sum_{i=1}^{2n-1}a_i^ha_i^m\sum_{j=1}^{2n-1}a_j^ka_j^l\\
&=\sum_{h,k,l,m=1}^{2n}Z_h(\vh_k)Z_l(\vh_m)(\delta_{hm}-\vh_h\vh_m)(\delta_{kl}-\vh_k\vh_l)\\
&=\sum_{h,k=1}^{2n}Z_h(\vh_k)Z_k(\vh_h).
            \end{split}
        \end{equation*}
        Exploiting \Cref{norofh}, the thesis follows.
    \end{proof}
    In view of \Cref{norofh} and \Cref{normoftildeh}, non-characteristic hypersurfaces of class $C^2$ with $h\equiv 0$ are trivially vertical hyperplanes, provided that $n\geq 2$. Indeed, if $S$ is such a hypersurface, $N$ is its Euclidean unit normal and $\vh$ its horizontal unit normal, then \Cref{normoftildeh} and \eqref{propvh4} imply that $\tilde h\equiv 0$, $N_{2n+1}\equiv 0$, $N=N(\bar x,\bar y)$ and $\vh=(N_1,\ldots,N_{2n})$. Hence
    \begin{equation*}
        0=|\tilde h|^2=\sum_{i,j=1}^{2n}Z_i\vh_jZ_j\vh_i=\sum_{i,j=1}^{2n+1}\frac{\partial N_j}{\partial z_i}\frac{\partial N_i}{\partial z_j},
    \end{equation*}
    where the last term coincides with the squared norm of the Euclidean second fundamental form of $S$. Hence $S$ is a hyperplane, which is vertical since $N_{2n+1}\equiv 0$.
    As already mentioned, when $n\geq 2$ it is not in general true that $\tilde h=0$ implies $h=0$.
    \begin{example}\label{esempio}
            As an instance, consider in $\hh^2$ the non-vertical hyperplane
    \begin{equation*}
        S:=\{(\bar x,\bar y,t)\in\hh^2\,:\,a_1x_1+a_2x_2+b_1y_1+b_2y_2+t+d=0\}
    \end{equation*}
    for some $a_1,a_2,b_1,b_2,d\in\rr$. An easy computation shows that 
    \begin{equation*}
        \no(p)=\frac{(a_1,a_2,b_1,b_2,1)}{\sqrt{1+a_1^2+a_2^2+b_1^2+b_2^2}}\qquad\text{and}\qquad\noh(p)=\frac{(a_1+y_1,a_2+y_2,b_1-x_1,b_2-x_2)}{\sqrt{1+a_1^2+a_2^2+b_1^2+b_2^2}}
    \end{equation*}
    for any $p\in S$. Therefore, $S$ has a unique characteristic point $p_0=(b_1,b_2,-a_1,-a_2,-d)$. Far from $p_0$, $\vh$ can be expressed by
    \begin{equation*}
        \vh(p)=\frac{(a_1+y_1,a_2+y_2,b_1-x_1,b_2-x_2)}{\sqrt{(a_1+y_1)^2+(a_2+y_2)^2+(b_1-x_1)^2+(b_2-x_2)^2}}
    \end{equation*}
    for any $p\in S\setminus S_0$.
    \end{example}
Recalling \eqref{propvh4}, a tedious but simple computations shows that
\begin{equation*}
\sum_{h,k=1}^4Z_h(\vh_k)Z_k(\vh_h)=-\frac{2}{(a_1+y_1)^2+(a_2+y_2)^2+(b_1-x_1)^2+(b_2-x_2)^2}=-2(Td^\hh)^2.
\end{equation*}
Hence, \Cref{normoftildeh} implies that $\tilde h\equiv 0$ on $S\setminus S_0$. Nevertheless, in view of the previous computation and \Cref{norofh}, we conclude that
\begin{equation*}
    |h_p|^2=\frac{2}{(a_1+y_1)^2+(a_2+y_2)^2+(b_1-x_1)^2+(b_2-x_2)^2}
\end{equation*}
for any $p\in S\setminus S_0$.

\section{Local existence of geodesics on hypersurfaces} \label{geosection}
Let $S$ be a hypersurface of class $C^2$. Let $p\in S\setminus S_0$ and $w\in \hhh T_pS$. We wish to find a curve $\Gamma\in C^2(I,S)$ solving the differential problem
\begin{equation}\label{geoS}
    \left\{
\begin{aligned}
&\Gamma\text{ is horizontal}\\
&\nabla^S_{\Dot{\Gamma}}\Dot{\Gamma}=0&&\text{on }I\\
&
\Gamma(0)=p\\
&\Dot{\Gamma}(0)=w
\end{aligned}
\right.
\end{equation}
Arguing for instance as in \cite{MR2609016}, it is not difficult to show that solutions to \eqref{geoS} are geodesics in the Carnot-Carathéodory space associated with the sub-Riemannian structure $(S,\langle\cdot,\cdot\rangle_S)$. 
First, notice that, by means of \cite[Theorem 6.5]{MR1871966} and \cite[Theorem 1.2]{MR2223801} and without loss of generality, there exists $\Om\subseteq\rr^{2n}$ and $\varphi\in C(\Om)$ such that $\nabla^\varphi\varphi\in C(\Om,\rr^{2n-1})$, $U=i(\Om)\cdot j(\rr)$ is an open neighborhood of $p$ and 
\begin{equation*}
    S\cap U=\graf_{Y_1}(\varphi,\Om)\cap U.
\end{equation*}
We need the following lemma. 
\begin{proposition}\label{fireg}
Let $\varphi\in C(\Om)$ be such that $\nabla^\varphi\varphi\in C(\Om,\rr^{2n-1})$. Assume that $\graf_{Y_1}(\varphi,\Om)$ is a non-characteristic hypersurface of class $C^2$. Then  $\varphi\in C^{2}(\Om)$.
\end{proposition}
\begin{proof}
    Let us consider the map $g:\hh^n\longrightarrow\hh^n$ defined by
    \begin{equation*}
        g(\bar x,\bar y,t)=(\bar x,\bar y,t-x_1y_1)
    \end{equation*}
    for any $(\bar x,\bar y,t)\in\mathbb H^n$. Notice that $g$ is smooth, bijective and and $\det(Dg)\equiv 1$. Hence $g$ is a smooth diffeomorphism. Let us set $\hat{S}:=g(S)$. Notice that $\hat S$ is of class $C^{2}$. It is easy to check that
    \begin{equation*}
        \hat{S}\cap g(U)=\{(\bar\xi,\varphi(\bar\xi,\tilde\eta,\tau),\tilde\eta,\tau)\,:\,(\tilde\xi,\bar\eta,\tau)\in \Om\}.
    \end{equation*}
    Therefore the thesis follows provided that $(\hat{N}(\hat{p}))_{n+1}\neq 0$ for any $\hat{p}\in\hat{S}\cap g(U)$, being $\hat{N}(\hat{p})$ the Euclidean normal to $\hat{S}$ at $\hat{p}$. Assume by contradiction that there exists $\hat{p}\in\hat{S}\cap g(U)$ such that $(\hat{N}(\hat{p}))_{n+1}=0$. This implies that $(\bar 0,1,\tilde 0,0)\in T_{\hat{p}} \hat{S}$. Let $p\in S$ be such that $g(p)=\hat{p}$.
    Noticing that
    \begin{equation*}
        (dg)|_p(Y_1|_p)=(\bar 0,1,\tilde 0,0)\in T_{\hat{p}} \hat{S},
    \end{equation*}
    we infer that $Y_1|_p\in T_p S$. Since $S$ is non-characteristic, \eqref{hornormnonchar} implies that $(\vh(p))_{n+1}=0$. On the other hand, we know from \cite[Theorem 1.2]{MR2223801} that $(\vh(p))_{n+1}\neq 0$, a contradiction.
    
\end{proof}
Therefore we reduce \eqref{geoS} to a differential problem for curves in $\Om$. To this aim, fix $q\in\Om$ such that $\Psi(q)=p$, and let $\gamma(s)=(\bar\xi(s),\tilde\eta(s),\tau(s)):I\longrightarrow\Om$. If we lift $\gamma$ to a curve $\Gamma:I\longrightarrow\hn$ by letting
\begin{equation*}
    \Gamma(s)=\Psi(\gamma(s))=(\bar\xi(s),\varphi(\gamma(s)),\tilde\eta(s)),\tau(s)-\xi_1(s)\varphi(\gamma(s)))
\end{equation*}
for any $s\in I$, then by construction $\Gamma(I)\subseteq S$. From now on, we fix the notation $\alpha(s):=\varphi(\gamma(s))$. To give a meaning to \eqref{geoS} we need that $\Dot{\Gamma}$ is horizontal. Notice that
\begin{equation*}
\begin{split}
    \Dot{\Gamma}&=(\Dot{\xi_1},\ldots,\Dot{\xi_n},\Dot{\alpha},\Dot{\eta_2},\ldots,\Dot{\eta_n},\Dot{\tau}-\Dot{\xi_1}\alpha-\xi_1\Dot{\alpha})\\
    &=\sum_{j=1}^n\Dot{\xi_j}X_j+\Dot{\alpha}Y_1+\sum_{2=1}^n\Dot{\eta_j}Y_j+\left(\Dot{\tau}-2\alpha\Dot{\xi_1}-\sum_{j=2}^n\eta_j\Dot{\xi_j}+\sum_{j=2}^n\xi_j\Dot{\eta_j}\right)T.
\end{split}
\end{equation*}
Therefore $\Dot{\Gamma}$ admits a $C^1$ extension to the whole $\hhh TS$ if and only if 
\begin{equation}\label{uno}
    \Dot{\tau}=2\alpha\Dot{\xi_1}+\sum_{j=2}^n\eta_j\Dot{\xi_j}-\sum_{j=2}^n\xi_j\Dot{\eta_j},
\end{equation}
that is if and only if $\gamma$ is horizontal in $(\Om,d_\varphi)$.
Let us denote such an extension by $
A=\sum_{j=1}^{2n}\psi_j Z_j$. This means that $A\in C^1(S,\hhh TS)$ and \begin{equation*}
    \psi_{j}(\Gamma(s))=\Dot{\Gamma}_j(s)
\end{equation*}
for any $s\in I$ and any $j=1,\ldots,2n$.
Thanks to the aforementioned properties of $\nabla^ S$ and recalling \eqref{phflat}, then 
\begin{equation*}
\begin{split}
  \nabla^S_{\Dot{\Gamma}}\Dot{\Gamma}\big|_{\Gamma(s)}&=\nabla_{\Dot{\Gamma}}\Dot{\Gamma}\big|_{\Gamma(s)}-\left\langle\nabla_{\Dot{\Gamma}}\Dot{\Gamma}\big|_{\Gamma(s)},\vh\big|_{\Gamma(s)}\right\rangle\vh\big|_{\Gamma(s)}\\
  &=\sum_{j=1}^{2n}\langle\Dot{\Gamma}(s),\nabla_H\psi_j(\Gamma(s))\rangle Z_j\big|_{\Gamma(s)}-\left(\sum_{k=1}^{2n}\langle\Dot{\Gamma}(s),\nabla_H\psi_j(\Gamma(s))\rangle\vh_k\big|_{\Gamma(s)}\right)\vh\big|_{\Gamma(s)}\\
  &=\sum_{j=1}^{2n}\Ddot{\Gamma}_j(s)Z_j\big|_{\Gamma(s)}-\left(\sum_{k=1}^{2n}\Ddot{\Gamma}_k(s)\vh_k\big|_{\Gamma(s)}\right)\vh\big|_{\Gamma(s)}
\end{split}
\end{equation*}
for any $s\in I$. Hence
$\nabla^S_{\Dot{\Gamma}}\Dot{\Gamma}=0$ if and only if
    \begin{equation}\label{sistGamma}
        \Ddot{\Gamma}_j-\vh_j\langle\Ddot{\Gamma},\vh\rangle=0
    \end{equation}
    for any $j=1,\ldots,2n$.
We need to traduce \eqref{sistGamma} in terms of $\gamma$. To this aim, recalling \eqref{uno}, notice that
\begin{equation*}
    \Ddot{\Gamma}=\sum_{j=1}^n\Ddot{\xi_j}X_j+\Ddot{\alpha}Y_1+\sum_{2=1}^n\Ddot{\eta_j}Y_j.
\end{equation*}

\begin{lemma}\label{aux1}
It holds that
\begin{equation*}
      \langle\Ddot{\Gamma},\vh\rangle=-W^{-\frac{1}{2}}\left(2\tilde T\varphi\Dot\alpha\Dot\xi_1+\langle D^2\varphi\Dot\gamma,\Dot\gamma\rangle\right).
\end{equation*}
\end{lemma}
\begin{proof}
Notice that
\begin{equation}\label{alfacomp}
    \Dot\alpha(s)=\langle\Dot\gamma,D\varphi(\gamma(s))\rangle\qquad\text{and}\qquad\Ddot\alpha(s)=\langle\Ddot\gamma(s),D\varphi(\gamma(s))\rangle+\langle D^2\varphi(\gamma(s))\Dot\gamma(s),\Dot\gamma(s)\rangle
\end{equation}
for any $s\in I$. Moreover, taking derivatives in \eqref{uno}, we see that
\begin{equation}\label{tauduep}
    \Ddot{\tau}=2\Dot{\alpha}\Dot{\xi}_1+2\alpha\Ddot{\xi}_1+\sum_{j=2}^n\eta_j\Ddot{\xi}_j-\sum_{j=2}^n\xi_j\Ddot{\eta}_j.
\end{equation}
    Exploiting \eqref{normygraph}, \eqref{alfacomp} and \eqref{tauduep}, we see that
    \begin{equation*}
        \begin{split}
W^{\frac{1}{2}}\langle\Ddot{\Gamma},\vh\rangle&=W^\varphi\varphi\Ddot\xi_1+\sum_{j=2}^{n}\tilde X_j\varphi\Ddot\xi_j+\sum_{j=2}^n\tilde Y_j\varphi\Ddot{\eta}_j-\Ddot\alpha\\
&=\Ddot\xi_1\varphi_{\xi_1}+2\Ddot\xi_1\alpha\varphi_\tau+\sum_{j=2}^n\Ddot\xi_j\varphi_{\xi_j}+\sum_{j=2}^n\eta_j\Ddot\xi_j\varphi_{\tau}+\sum_{j=2}^n\Ddot\eta_j\varphi_{\eta_j}-\sum_{j=2}^n\xi_j\Ddot\eta_j\varphi_{\tau}\\
&\quad-\Ddot\xi_1\varphi_{\xi_1}-\sum_{j=2}^n\Ddot\xi_j\varphi_{\xi_j}-\sum_{j=2}^n\Ddot\eta_j\varphi_{\eta_j}-\Ddot\tau\varphi_{\tau}-\langle D^2\varphi\Dot\gamma,\Dot\gamma\rangle\\
&=\tilde T\varphi\left(2\alpha\Ddot\xi_1+\sum_{j=2}^n\eta_j\Ddot\xi_j-\sum_{j=2}^n\Ddot\eta_j\xi_j-\Ddot\tau\right)-\langle D^2\varphi\Dot\gamma,\Dot\gamma\rangle\\
&=-2\tilde T\varphi\Dot\alpha\Dot\xi_1-\langle D^2\varphi\Dot\gamma,\Dot\gamma\rangle.
        \end{split}
    \end{equation*}
\end{proof}
In the following, we let $M=2\tilde T\varphi\Dot\alpha\Dot\xi_1+\langle D^2\varphi\Dot\gamma,\Dot\gamma\rangle$. Notice that, by \Cref{aux1}, the term $\langle\Ddot\Gamma,\vh\rangle$ does not involve second derivatives of $\gamma$. Therefore \eqref{geoS} is equivalent to the following differential problem.
\begin{equation}\label{geoom}
    \left\{
\begin{aligned}
&\Ddot\xi_1+W^{-1}W^\varphi\varphi M=0&&\text{on }I,\qquad\xi_1(0)=\xi_1^0,\qquad\Dot\xi_1(0)=w_1\\
&\Ddot\xi_j+W^{-1}\tilde X_j\varphi M=0&&\text{on }I,\qquad\xi_j(0)=\xi_j^0,\qquad\Dot\xi_j(0)=w_j\qquad\text{$j=2,\ldots,n$}\\
&\Ddot\alpha-W^{-1}M=0&&\text{on }I,\qquad\alpha(0)=y_1,\qquad\Dot\alpha(0)=w_{n+1}\\
&\Ddot\eta_j+W^{-1}\tilde Y_j\varphi M=0&&\text{on }I,\qquad\eta_j(0)=\eta_j^0,\qquad\Dot\eta_j(0)=w_{n+j}\qquad\text{$j=2,\ldots,n$}\\
&\Dot{\tau}=2\alpha\Dot{\xi_1}+\sum_{j=2}^n\eta_j\Dot{\xi_j}-\sum_{j=2}^n\xi_j\Dot{\eta_j}&&\text{on }I,\qquad \tau(0)=t+\xi_1^0\varphi(q)
\end{aligned}
\right.
\end{equation}
A key step consist in showing that the third line of \eqref{geoom} is redundant.
\begin{lemma}\label{auxmain}
    A curve $\gamma\in C^2(I,\Om)$ solves \eqref{geoom} if and only if it solves the following differential system.
    \begin{equation}\label{geoomrid}
    \left\{
\begin{aligned}
&\Ddot\xi_1+W^{-1}W^\varphi\varphi M=0&&\text{on }I,\qquad\xi_1(0)=\xi_1^0,\qquad\Dot\xi_1(0)=w_1\\
&\Ddot\xi_j+W^{-1}\tilde X_j\varphi M=0&&\text{on }I,\qquad\xi_j(0)=\xi_j^0,\qquad\Dot\xi_j(0)=w_j\qquad\text{$j=2,\ldots,n$}\\
&\Ddot\eta_j+W^{-1}\tilde Y_j\varphi M=0&&\text{on }I,\qquad\eta_j(0)=\eta_j^0,\qquad\Dot\eta_j(0)=w_{n+j}\qquad\text{$j=2,\ldots,n$}\\
&\Dot{\tau}=2\alpha\Dot{\xi_1}+\sum_{j=2}^n\eta_j\Dot{\xi_j}-\sum_{j=2}^n\xi_j\Dot{\eta_j}&&\text{on }I,\qquad \tau(0)=t+\xi_1^0\varphi(q)
\end{aligned}
\right.
\end{equation}
\end{lemma}
\begin{proof}
    If $\gamma\in C^2(I,\Om)$ solves \eqref{geoom}, then clearly solves \eqref{geoomrid}. Conversely, assume that $\gamma\in C^2(I,\Om)$ solves \eqref{geoomrid}. Since $y_1=\varphi(q)$, then $\alpha(0)=\varphi(\gamma(0))=\varphi(q)=y_1$. Moreover, notice that
    \begin{equation*}
        \begin{split}
            \Dot\alpha(0)&=\langle\Dot\gamma(0),D\varphi(q)\rangle\\
            &=\Dot\xi_1(0)\varphi_{\xi_1}(q)+\sum_{j=2}^n\Dot\xi_j(0)\varphi_{\xi_j}(q)+\sum_{j=2}^n\Dot\eta_j(0)\varphi_{\eta_j}(q)+\Dot\tau(0)\varphi_\tau(q)\\
            &=w_1W^\varphi\varphi(q)+\sum_{j=2}^nw_j\tilde X_j\varphi(q)+\sum_{j=2}^nw_{n+j}\tilde Y_j\varphi(q)\\
            &=w_{n+1},
        \end{split}
    \end{equation*}
    where the last equality follows from \eqref{normygraph} and the fact that $w\in \hhh T_pS$. Observe that, recalling \eqref{tauduep} and exploiting all the second-order equations in \eqref{geoomrid},
    \begin{equation*}
        \begin{split}
            \langle\Ddot\gamma,D\varphi\rangle&=\Ddot\xi_1\varphi_{\xi_1}+\sum_{j=2}^n\Ddot\xi_j\varphi_{\xi_j}+\sum_{j=2}^n\Ddot\eta_j\varphi_{\eta_j}+\Ddot\tau\Tilde T\varphi\\
            &=\Ddot\xi_1 W^\varphi\varphi+\sum_{j=2}^n\Ddot\xi_j\tilde X_j\varphi+\sum_{j=2}^n\Ddot\eta_j\tilde Y_j\varphi+2\Tilde T\varphi\Dot\alpha\Dot\xi_1\\
            &=-W^{-1}M|\nabla^\varphi\varphi|^2+2\Tilde T\varphi\Dot\alpha\Dot\xi_1.
        \end{split}
    \end{equation*}
    Therefore, we conclude that
    \begin{equation*}
        \begin{split}
            \Ddot\alpha-W^{-1}M&=W^{-1}(W\langle\Ddot\gamma,D\varphi\rangle+W\langle D^2\varphi\Dot\gamma,\Dot\gamma\rangle-2\tilde T\varphi\Dot\alpha\Dot\xi_1-\langle D^2\varphi\Dot\gamma,\Dot\gamma\rangle)\\
            &=W^{-1}(\langle\Ddot\gamma,D\varphi\rangle+|\nabla^\varphi\varphi|^2\langle\Ddot\gamma,D\varphi\rangle+|\nabla^\varphi\varphi|^2\langle D^2\varphi\Dot\gamma,\Dot\gamma\rangle-2\tilde T\varphi\Dot\alpha\Dot\xi_1)\\
            &=W^{-1}(-W^{-1}M|\nabla^\varphi\varphi|^2+|\nabla^\varphi\varphi|^2\langle\Ddot\gamma,D\varphi\rangle+|\nabla^\varphi\varphi|^2\langle D^2\varphi\Dot\gamma,\Dot\gamma\rangle)\\
            &=\frac{|\nabla^\varphi\varphi|^2}{1+|\nabla^\varphi\varphi|^2}(\Ddot\alpha-W^{-1}M),
        \end{split}
    \end{equation*}
    which is equivalent to say that
    \begin{equation*}
        W^{-1}(\Ddot\alpha-W^{-1}M)=0.
    \end{equation*}
    Being $W^{-1}\neq 0$, the thesis follows.
\end{proof}
We can summarize the previous achievements in the following statement. 
\begin{proposition}\label{equiprop}
    The following properties hold.
    \begin{itemize}
        \item [$(i)$] If $\Gamma\in C^2(I,S)$ solves \eqref{geoS}, then $\gamma:I\longrightarrow\Om$ defined by \begin{equation*}
            \gamma(s):=\Pi(\Gamma(s))
        \end{equation*}
        for any $s\in I$ solves \eqref{geoomrid}.
        \item [$(ii)$] If $\gamma\in C^2(\Om)$ solves \eqref{geoomrid}, then $\Gamma:I\longrightarrow\Om$ defined by \begin{equation*}
            \Gamma(s):=\Psi(\gamma(s))
        \end{equation*}
        for any $s\in I$ solves \eqref{geoS}.
    \end{itemize}
\end{proposition}
\begin{proof}
    $(ii)$ follows thanks to \Cref{auxmain}. To prove $(i)$, notice that, if $\Gamma$ is as in the statement, then $\Gamma=\Psi(\sigma)$, where $\sigma=\Pi(\Gamma)$, and so $(i)$ easily follows.
\end{proof}
\begin{theorem}\label{exungeo}
    The initial value  problem \eqref{geoS} admits a unique local solution $\Gamma\in C^2(I,S)$.
\end{theorem}
\begin{proof}
    In view of \Cref{equiprop}, it suffices to show that the initial value problem \eqref{geoomrid} admits locally a unique solution. Notice that \eqref{geoomrid} can be seen as a fist-order initial value problem by means of a standard doubling variable argument. More precisely, let us introduce the equations
\begin{equation}\label{neweq}
        \Dot\xi_1=\Xi_1,\qquad\Dot\xi_j=\Xi_j\qquad\text{and}\qquad\Dot\eta_j=H_j
    \end{equation}
    for any $j=2,\ldots,n$, let us define the curve $\tilde\Gamma:I\longrightarrow\rr^{4n-1}$ by
    \begin{equation*}
        \tilde\Gamma=(\xi_1,\ldots,\xi_n,\eta_2,\ldots,\eta_n,\tau,\Xi_1,\ldots,\Xi_n,H_2,\ldots,H_n),
    \end{equation*}
    and let $\tilde q=(x_1,\ldots,x_n,y_2,\ldots,y_n,t-x_1y_1,w_1,\ldots,w_n,w_{n+2},\ldots,w_{2n})$. Then \eqref{geoomrid} is equivalent to the first-order initial value problem
\begin{equation}\label{dptau}
    \left\{
\begin{aligned}
&\Dot{\tilde\Gamma}(s)=F(s,\Gamma(s))\quad\text{on }I\\
&\tilde\Gamma(0)=\tilde q
\end{aligned}
\right.
\end{equation}
where $F:I\times\rr^{4n-1}\longrightarrow\rr^{4n-1}$ is defined in the obvious way taking into account \eqref{geoomrid} and \eqref{neweq}. Thanks to \Cref{fireg}, $F$ is of class $C^1$ in a neighborhood of $(0,\tilde q)$. Hence the thesis follows by means of the classical Picard-Lindel\"of Theorem (cf. e.g. \cite{MR1929104}).
\end{proof}

\begin{proof}[Proof of \Cref{main1}]
    Fix $p=(\bar x,\bar y,t)\in S\setminus S_0$.
    Assume first that there exists an open neighborhood $U$ of $p$ such that $\tilde h\equiv 0$ on $U$. Fix $w\in \hhh T_p S$. As before, recalling also \Cref{fireg}, we can assume that there exists $\Om\subseteq\rr^{2n}$ and $\varphi\in C^2(\Om)$ such that 
\begin{equation*}
    S\cap V=\graf_{Y_1}(\varphi,\Om)\cap V,
\end{equation*}
where $V=i(\Om)\cdot j(\rr)$. In view of \Cref{exungeo}, there exists a small domain $I\subseteq\rr$ such that $0\in I$ and a curve $\Gamma\in C^2(I,S)$ solving \eqref{geoS} with initial data $\Gamma(0)=p$ and $\Dot\Gamma(0)=w$. Since $\tilde h\equiv 0$, and recalling \Cref{charinhn}, we conclude that $\Gamma(s):=p\cdot(sw,0)$. Hence $S$ is locally ruled at $p$. Conversely, assume that $S$ is locally ruled in a suitable neighborhood $U$ of $p$. Assume also that $U\cap S_0=\emptyset$. Fix $\bar p\in U$ and $w\in \hhh T_{\bar p}S$. Since $\Gamma(s):=\bar p\cdot(sw,0)$ lies locally in $S$, then $h_{\bar p}(w,w)=0$, and so $\tilde h_{\bar p}=0$.
\end{proof}

\begin{proof}[Proof of \Cref{mainmain}]
    The first equivalence follows from \Cref{rulcar}. If in addition $S$ is topologically closed and $n\geq 2$, arguing as in \cite[Proposition 4.1]{MR3794892} it is easy to see that the fact that $S$ is horizontally totally geodesic implies that $S_0$ is constituted by isolated points, and so it is countable. The thesis then follows by \Cref{main1}.
\end{proof}
\begin{proof}[Proof of Theorem \ref{controes}]
$S$ is clearly a smooth hypersurface. Let $p\in S\setminus S_0$. It is well-known that 
    \begin{equation*}
        N(p)=\frac{1}{\sqrt{1+|D u(z)|_{\Ru^{2n}}^2}}(D u(z),-1)=\frac{1}{\sqrt{1+x_1^2+y_1^2}}(x_1,0,\ldots,0,-y_1,0,\ldots,0,-1),
    \end{equation*}
    and so
    \begin{equation*}
        \vh(p)=\vh(z)=\frac{1}{\sqrt{2(x_1-y_1)^2+\sum_{j=2}^n(x_j^2+y_j^2)}}(x_1-y_1,-y_2,\ldots,-y_n,x_1-y_1,x_2,\ldots,x_n).
    \end{equation*}
    Since in this case $\vh$ does not depend on $t$, an easy computation shows that  
    \begin{equation}\label{diverfacile}
        \diver_{\mathbb H}\vh(p)=\diver_{\Ru^{2n}}\vh(z)=0
    \end{equation}
    for any $p\in S\setminus S_0$. Since $n\geq 2$, \eqref{diverfacile} allows us to apply \cite[Corollary F]{MR2262784} and \cite[Theorem 2.3]{MR2333095}, which, together with \cite[Example 5.29]{MR3587666}, imply that $S$ is minimal. We conclude noticing that, in view of Theorem \ref{mainmain}, $S$ is not horizontally totally geodesic.
\end{proof}

\bibliographystyle{abbrv}
\bibliography{biblio}
\end{document}